\documentclass{amsart}[12pt]
\usepackage{amsmath}
\usepackage{amssymb}
\usepackage{mathrsfs}
\usepackage{amsfonts}
\usepackage{graphicx}
\usepackage{texdraw}
\usepackage{graphpap}
\usepackage{color,bbm}
\usepackage{wasysym}
\usepackage{enumitem}
\usepackage{subcaption}
\usepackage{pxfonts}
\usepackage[OT2,T1]{fontenc}

\input txdtools

\newtheorem{thm}{Theorem}[section]
\newtheorem{cor}[thm]{Corollary}

\newtheorem{lem}[thm]{Lemma}

\newtheorem{prop}[thm]{Proposition}

\theoremstyle{definition}

\newtheorem{conj}[thm]{Conjecture}

\theoremstyle{remark}
\newtheorem*{rem}{Remark}

\numberwithin{equation}{section}
\numberwithin{figure}{section}

\newcommand{\lr}{\longrightarrow}

\newcommand{\calW}{{\mathscr W}}
\newcommand{\bfB}{{\mathbf B}}
\newcommand{\bfK}{{\mathbf K}}
\newcommand{\bfP}{{\mathbf P}}
\newcommand{\bfE}{{\mathbf E}}

\newcommand{\trace}{\operatorname{trace}}
\newcommand{\erfc}{\operatorname{erfc}}

\newcommand{\R}{{\mathbb R}}
\newcommand{\bfR}{{\mathbf R}}

\newcommand{\C}{{\mathbb C}}

\newcommand{\magn}{{\Gamma}}

\newcommand{\eps}{{\varepsilon}}

\newcommand{\re}{\operatorname{Re}}
\newcommand{\im}{\operatorname{Im}}

\newcommand{\Int}{\operatorname{Int}}

\newcommand{\perim}{\operatorname{perim}}

\renewcommand{\P}{{\mathbf{P}}}

\renewcommand{\d}{{\partial}}
\newcommand{\dbar}{\overline{\partial}}
\newcommand{\1}{{\mathbf 1}}

\newcommand{\lspan}{\operatorname{span}}
\newcommand{\dist}{\operatorname{dist}}
\newcommand{\supp}{\operatorname{supp}}

\newcommand{\seq}{a}
\newcommand{\newM}{A}

\newcommand{\newA}{A}

\newcommand{\bfp}{\boldsymbol{p}}
\renewcommand{\L}{\mathbb{L}}
\newcommand{\vt}{\vartheta}
\newcommand{\family}{\boldsymbol{\zeta}}

\newcommand{\conf}{\{\zetaj\}_1^n}

\newcommand{\sepn}{\mathrm{s}}
\newcommand{\sep}{\mathrm{s}}
\newcommand{\dens}{D_{BL}}
\newcommand{\fbf}{F}

\newcommand{\zetak}{\zeta_k}
\newcommand{\zetaj}{\zeta_j}

\makeatletter
\newcommand*\bigcdot{\mathpalette\bigcdot@{.5}}
\newcommand*\bigcdot@[2]{\mathbin{\vcenter{\hbox{\scalebox{#2}{$\m@th#1\bullet$}}}}}
\makeatother

\begin{document}

\title[The low temperature Coulomb gas]{The planar low temperature Coulomb gas: separation and equidistribution}



\author{Yacin Ameur}
\address{Yacin Ameur\\
Department of Mathematics\\
Lund University\\
22100 Lund, Sweden\\
and Erwin Schr\"odinger International Institute for Mathematics and Physics, University of Vienna,
Boltzmanngasse 9A, A-1090 Vienna, Austria.}
\email{ Yacin.Ameur@math.lu.se}

\author{Jos\'e Luis Romero}
\address{Jos\'e Luis Romero\\Faculty of Mathematics \\
University of Vienna \\
Oskar-Morgenstern-Platz 1 \\
A-1090 Vienna, Austria\\ and
Acoustics Research Institute\\ Austrian Academy of
Sciences\\Wohllebengasse 12-14, A-1040 Vienna, Austria\\ and Erwin Schr\"odinger International Institute for Mathematics and Physics, University of Vienna,
Boltzmanngasse 9A, A-1090 Vienna, Austria.}
\email{jose.luis.romero@univie.ac.at}

\begin{abstract}

We consider planar Coulomb systems consisting of a large number $n$ of repelling point charges in the low temperature regime, where the inverse temperature $\beta$ grows at least logarithmically in $n$
as $n\lr\infty$, i.e., $\beta\gtrsim \log n$.

Under suitable conditions on an external potential
we prove results to the effect that the gas is with high probability uniformly separated and equidistributed with respect to the corresponding equilibrium measure (in the
given external field).

Our results generalize earlier results about Fekete configurations, i.e., the case $\beta=\infty$.
There are also several auxiliary results which could be of independent interest. For example, our method of proof of equidistribution (a variant of ``Landau's method'')
 works for general families of configurations which are uniformly separated and which satisfy
certain sampling and interpolation inequalities.
\end{abstract}

\thanks{Y. A. and J. L. R. gratefully acknowledge support from the Austrian Science Fund (FWF): Y 1199 and from the Research in Teams programme of the Erwin Schr\"odinger International Institute for Mathematics and Physics of University of Vienna (``Time-frequency analysis of random point processes'').}

\maketitle

\section{Introduction}
\subsection{Main results}
Let us briefly recall the setting of the planar Coulomb gas with respect to an external potential $Q$ in the plane and an inverse temperature $\beta=1/(k_BT)>0$.

The potential $Q$ is a fixed function from the complex plane $\C$ to $\R\cup\{+\infty\}$.
It is always assumed that $Q$ is lower semicontinuous, is finite on some set of positive capacity, and obeys the growth condition
\begin{align}\label{eq_qlarge}
\liminf_{\zeta \lr \infty} \frac{Q(\zeta)}{2 \log |\zeta|} > 1.
\end{align}

To a plane
configuration $\{\zetaj\}_1^n\subset\C$ we then associate the Hamiltonian (or total energy)
$$H_n=\sum_{j\ne k}\log\frac 1 {|\zetaj-\zetak|}+n\sum_{j=1}^n Q(\zetaj),$$
and form the Boltzmann-Gibbs measure on $\C^n$
\begin{equation}\label{bg}d\P_n^{\,\beta}=\frac 1 {Z_n^{\,\beta}}e^{\,-\beta H_n}\, dA_n.\end{equation}
(The constant $Z_n^{\,\beta}$ is chosen so that $\P_n^{\,\beta}$ is a probability measure.)

Here and throughout we use the convention that ``$dA$'' denotes the Lebesgue measure in $\C$ divided by $\pi$, i.e., $dA=\tfrac 1 \pi \,dx\, dy$.
We write $dA_n$ for the product measure on $\C^n$: $dA_n=(dA)^{\otimes n}.$

A configuration $\conf$ which renders $H_n$ minimal is known as a Fekete configuration. In a sense Fekete configurations correspond to the inverse temperature $\beta=\infty$.

In the paper \cite{A2}, a related \emph{low temperature regime} was studied, when the inverse temperature increases at least logarithmically with the number $n$ of particles, i.e.,
\begin{equation}\label{freezing}\beta=\beta_n\ge c\log n,\end{equation}
where $c$ is an arbitrary but fixed, strictly positive number.

In the present work, we shall find further support for the picture
that \eqref{freezing} gives a natural ``freezing regime'' as the parameter $c$ increases from $0$ to $\infty$, in the sense that the system becomes more and more ``lattice-like'' in this transition.

We now recall some results from classical potential theory that can be found in \cite{HM} and \cite{ST}, for example.
For a given compactly supported Borel probability measure $\mu$ on $\C$, we define its logarithmic $Q$-energy by
$$I_Q[\mu]=\int_{\C^2}\log\frac 1 {|\zeta-\eta|}\, {d\mu(\zeta)\, d\mu(\eta)}+\mu(Q),$$
where $\mu(Q)$ is short for $\int Q\, d\mu$.

Under the above hypotheses there is a unique probability measure
$\sigma=\sigma[Q]$ which minimizes $I_Q$ over all compactly supported Borel probability measures, see \cite{ST}. This measure $\sigma$ is known as the \textit{equilibrium measure} in external potential $Q$,
and its support $S=\supp\sigma$ is called the
\emph{droplet}.

We will assume throughout that $Q$ is $C^2$-smooth in a neighborhood of $S$. This implies (by Frostman's theorem) that $\sigma$
is absolutely continuous and takes the form
$$d\sigma=\1_S\cdot\Delta Q\, dA,$$
where
$\Delta:=\d\dbar=\frac 1 4(\d_{xx}+\d_{yy})$
is one-quarter of the standard Laplacian. In particular, $\Delta Q\ge 0$ on the droplet $S$.

We remark that the system $\{\zeta_j\}_1^n$ tends to follow the equilibrium measure in the following sense. Let $\bfR_n^{\beta_n}(\zeta)$ be the usual 1-point
intensity function, i.e.,
\begin{equation}\label{onept}\bfR_n^{\,\beta_n}(\zeta)=\lim_{\eps\lr 0}\frac {\bfE_n^{\,\beta_n}(\#D(\zeta,\eps))}{\eps^{\,2}}.\end{equation}
Here and throughout we use the following terminology: if $B$ is a Borel subset of $\C$, then
$$\# B:=\#(B\cap\{\zeta_j\}_1^n)$$
denotes the number of particles $\{\zeta_j\}_1^n$ that fall in $B$.
Thus $\# B$ is an integer-valued random variable and $\bfR_n^{\,\beta_n}(\zeta)$ has the meaning of the expected number of particles per unit area at $\zeta$. Of course,
$D(\zeta,\eps)$ denotes the open disc with center $\zeta$ and radius $\eps$.

Recall  that $\tfrac 1 n\bfR_n^{\beta}\, dA\lr \sigma$ as $n\to\infty$ in the weak sense of measures, by the well-known Johansson equilibrium convergence theorem, \cite{J,HM}.
As noted in \cite[Theorem A.1]{A}, the proof of \eqref{joh00} for fixed $\beta$ in \cite{HM,J} works in the present situation if we assume (for example) a uniform lower bound
$\beta_n\ge\beta_0>0$, and if the entropy $\sigma(\log\Delta Q)$ is finite, i.e., we have
\begin{equation}\label{joh00}\tfrac 1 n\, \bfR_n^{\beta_n}\, dA\lr \sigma,\qquad (n\lr\infty)\end{equation}
in the weak sense of measures.

In what follows, it is convenient to impose the following (mild) assumptions on $Q$.
\begin{enumerate}[label=(\arabic*)]
\item \label{Q1} $\Delta Q>0$ in a neighborhood of the boundary $\d S$.
\item \label{Q2} The boundary $\d S$ has finitely many connected components.
\item \label{Q3} Each boundary component is an everywhere $C^1$-smooth Jordan curve.
\item \label{Q4} $S^*=S$ where $S^*$ is the coincidence set for the
obstacle problem associated with $Q$. (Concretely, this means that for $\zeta \in \mathbb{C} \setminus S$:
\begin{align*}
Q(\zeta) > \sup \big\{f(\zeta): f \in \mathcal{F}_Q \big\},
\end{align*}
where $\mathcal{F}_Q$ denotes the class of subharmonic functions on $\mathbb{C}$ that are everywhere $\leq Q$ and satisfy $f(\zeta) \leq \log|\zeta|^2 + O(1)$ as $|\zeta| \to \infty$.)
\end{enumerate}

See e.g.~\cite{ST} or \cite[Section 2]{A} for details about the obstacle problem associated with $Q$. It should be emphasized that some of the previous conditions are assumed merely for convenience. For example, condition \ref{Q4} could be avoided by
redefining the potential $Q$ to be $+\infty$
outside a small enough neighbourhood of the droplet. Also condition \ref{Q3} could be relaxed at the expense of some slight elaborations, but in the end those details have not seemed
interesting enough to merit inclusion in our present work.

Our goal is to study asymptotic properties of random samples $\{\zeta_j\}_1^n$
as $n \lr \infty$, in the low temperature regime \eqref{freezing}. The properties we have in mind
are conveniently expressed in terms of \emph{families} of configurations,
$$\family=(\family_n)_{n=1}^\infty,$$
where the configuration $\family_n=\{\zeta_{nj}\}_{j=1}^n$ is the $n$:th sample in the family. To lighten the notation, we usually write the $n$:th sample as $\{\zeta_j\}_1^n$ rather than
$\{\zeta_{nj}\}_1^n$.

We shall consider such families as picked randomly with respect to the product measure on $\prod_{n=1}^\infty \C^n$
\begin{equation}\label{eq_compat}
\P=\prod_{n=1}^\infty \bfP_n^{\,\beta_n},
\end{equation}
which we will likewise call a Boltzmann-Gibbs measure.

Given a plane configuration $\family_n=\{\zeta_j\}_1^n$, we define its (global, scaled)
\textit{spacing} by
\begin{equation}\label{separation}\sep_n(\family_n):=\sqrt{n}\cdot \min\big\{\,|\zetaj-\zetak|\,:\, j\ne k\,\big\}.\end{equation}
If $\sep_n(\family_n)\ge s_0>0$ we say that the configuration is \emph{$s_0$-separated}.
Similarly, a family $\family$  is said to be (asymptotically) \textit{$s_0$-separated} if
$$\liminf_{n \lr \infty} \sep_n(\family_n) \geq s_0.$$

The following theorem improves on a local separation result
from \cite{A2}, and also generalizes a global result for Fekete configurations in \cite{AOC}.

\begin{thm} \label{mth} (\emph{``Uniform separation''}) Let $Q$ be a $C^2$-smooth potential in a neighbourhood of the droplet satisfying \ref{Q1}-\ref{Q4}.
Also suppose that there is a constant $c>0$ such that
\begin{equation}\label{lt}\beta_n\ge c\log n.\end{equation}
Then there exists a constant $s_0=s_0(c)>0$ such that almost every family is $s_0$-separated, i.e.,
\begin{equation}\label{eq_s0}
\liminf_{n \lr \infty} \sep_n(\family_n)  \geq s_0,
\mbox{ almost surely.}
\end{equation}
\end{thm}

\begin{rem}
Our proof shows that \eqref{eq_s0} holds with (for example) $s_0=me^{-\frac 3 {2c}}$ where $m>0$ is a constant (depending only on $Q$).
\end{rem}

\begin{rem} In contrast to Theorem \ref{mth}, the separation result in \cite{A2} is local, valid near any point (bulk or boundary).
In the local setting, we may obtain
stronger bounds for the separation constant
depending on the strength of the Laplacian at the given point. (In particular, a substantial improvement is possible near a special point at which $\Delta Q=0$.)
Like in \cite{A2}, we may view Theorem \ref{mth} as a special case of a separation result valid for all $\beta$, not just for the low temperature regime; see a remark by the end
of Section \ref{sec_2}.
A local separation theorem for the bulk appeared also in the subsequent article \cite{ARSE}
(see part (4) of \cite[Theorem 1]{ARSE}),
depending on very different methods.

We shall find that Theorem \ref{mth} follows in a succinct way by
combining and developing ideas found in the recent works \cite{A,A2}.
\end{rem}

\begin{rem} By the Borel-Cantelli lemmas (see \cite{B}), our notion of almost sure convergence in \eqref{eq_s0} (with respect to $\P=\prod\bfP_n^{\,\beta_n}$)
is equivalent with that
\begin{equation}\label{our}\sum_{n=n_0}^\infty \bfP^{\,\beta_n}_n \left(\left\{\,\sep_n(\family_n) < s_0\, \right\}\right)\lr 0,\qquad (n_0\lr\infty).\end{equation}
This differs slightly from several related notions of convergence defined in Tao's book \cite[page 6]{T}.
For example, Tao would say that ``the event $\{\,\sep_n(\family_n)\ge s_0\,\}$ holds asymptotically almost surely as $n\to\infty$'' if
the convergence
$\lim_{n\to\infty}\bfP^{\,\beta_n}_n (\{\,\sep_n(\family_n) < s_0\, \})= 0$
holds. Likewise, Tao's notion of
``convergence with high probability'' is closely related to, but not quite the same, as \eqref{our}.
\end{rem}

We shall now address \textit{equidistribution} of random families in the low temperature regime.
For this purpose, it is convenient to impose stronger conditions on our potentials $Q$: we require in addition to our earlier assumptions that
\begin{enumerate}[label=(\arabic*)]
\setcounter{enumi}{4}
\item \label{Q5} $Q$ is real-analytic in a neighbourhood of $S$,
\item \label{Q6} $\Delta Q>0$ in a neighbourhood of $S$,
\item \label{Q7} $S$ is connected.
\end{enumerate}

\begin{rem}
An important consequence of condition \ref{Q5} is that it implies that the boundary $\d S$ is regular. Indeed, the well-known ``Sakai regularity theorem'' implies that under \ref{Q5} and \ref{Q6},
the boundary $\d S$ consists of finitely many analytic Jordan curves, possibly having finitely many singular points (cusps or double points) of known types. Such singular points
are precluded by condition \ref{Q3}. We shall freely
apply this result in the sequel. We refer to \cite[Section 6.3]{A0} as well as \cite{AKMW,LM} for details about the application of Sakai's theorem in the present setting. Sakai's original result,
which was formulated in a somewhat different way, is shown in \cite{GP,Sa1}, for example. Finally, it should be noted that the class of potentials meeting all requirements
\ref{Q1}-\ref{Q7} is very rich (one can begin with any element of a vast class of real-analytic functions and redefine it to be $+\infty$ near infinity \cite{EF,LM}). The paper \cite{LM} and the references there
contain many interesting examples, see also \cite{AKMW,BHa,BK,Sk,Te,Z}, for example.
\end{rem}

We next recall the notion of Beurling-Landau density of a family $\family=(\family_n)_n$ at a point $p$ in the plane (the ``zooming point''). It is advantageous to allow the zooming-point
to vary with $n$, i.e., $p=p_n$. We then look at the number of particles per unit area that fall in a microscopic disc about $p_n$ of radius $L/\sqrt{n}$, where $L>2$ is a (large) parameter.

We now come to the precise definition. Write $\bfp=(p_n)_1^\infty$ where $p_n$ are any points in the plane.
We
define the \emph{Beurling-Landau density} $\dens(\family,\bfp)$ of $\family$ at $\bfp$ by
\begin{align}\label{eq_BD}
\dens(\family,\bfp)=\lim_{L\lr\infty}\limsup_{n\lr\infty}\frac {\# D(p_n,L/\sqrt{n})}{L^{\,2}}=\lim_{L\lr\infty}\liminf_{n\lr\infty}\frac {\# D(p_n,L/\sqrt{n})}{L^{\,2}}
\end{align}
provided that the limits exist, and that the two expressions are indeed equal. (In general,
the two expressions in \eqref{eq_BD} are called
upper and lower densities.)

To express our next result, it is convenient to restrict attention to zooming points $p_n$ which converge to some limit $p_*\in \C$, i.e., $$p_n\lr p_*,\qquad  (n\lr\infty).$$

Following \cite{AOC} we say that $(p_n)_1^\infty$ belongs to the
\begin{itemize}
\item \textit{bulk regime} if $p_n\in \Int S$ for all $n$ and
$\sqrt{n}\,\dist (p_n,\C\setminus S)\lr\infty$ as $n\lr\infty$,
\item \textit{boundary regime} if $\limsup_{n\lr\infty}\sqrt{n}\,\dist (p_n,\d S)<\infty$,
 \item \textit{exterior regime} if $p_n\in \C\setminus S$ for all $n$ and
$\sqrt{n}\, \dist (p_n,S)\lr\infty$ as $n\lr\infty$.
\end{itemize}

As in \cite{A0} we could also include regimes near singular boundary points, but for reasons of length we shall here ignore this possibility (cf. assumption \ref{Q3}).

We can now state our second main result, which generalizes the equidistribution theorems for $\beta=\infty$ obtained in
\cite{AOC,A0}.

\begin{thm} \label{mth2} (\emph{``Equidistribution''}) Assume that $Q$ satisfies conditions \ref{Q1}-\ref{Q7} and that $\beta_n$ is in the low-temperature regime $\beta_n\ge c\log n$.
Then for almost every random family $\family$ from the corresponding Boltzmann-Gibbs distribution the following holds for every zooming point $\bfp=(p_n)$:
\begin{enumerate}[label=(\roman*)]
\item \label{pa1} If $\bfp$ belongs to the bulk regime then
$$\dens(\family,\bfp)=\Delta Q(p_*).$$
\item \label{pa3} If $\bfp$ belongs to the boundary regime then $$\dens(\family,\bfp)=\tfrac 1 2 \Delta Q(p_*).$$
\item \label{pa2} If $\bfp$ belongs to the exterior regime then
$$\dens(\family,\bfp)=0.$$
\end{enumerate}
Moreover, in each case, the convergence in $L$ towards the limit that defines the Beurling-Landau density \eqref{eq_BD} is uniform among all zooming sequences in the respective regimes, and among (almost) all families $\family$. (Cf. Proposition \ref{propop} for a more precise statement.)
\end{thm}

\begin{rem} We pause to discuss some context and the meaning of our result.
Excluding a zero-probability event, the following is true: given $\varepsilon>0$, there exists $L_0>0$ such that for any family $\family$, any zooming point $\bfp$ and any $L \geq L_0$,
\begin{align*}
l-\eps\le \liminf_{n\lr\infty}\frac {\# D(p_n,L/\sqrt{n})}{L^{\,2}}\le \limsup_{n\lr\infty}\frac {\# D(p_n,L/\sqrt{n})}{L^{\,2}}\le l+\eps,
\end{align*}
where $l=\Delta Q(p_*)$, $\tfrac12 \Delta Q(p_*)$, or $0$ depending on the regime of $\bfp$.

The zooming sequence $\bfp=(p_n)$ is thus allowed to be family-dependent,
and $p_n$ may for instance track the region where $\family_n$ is most concentrated.

A family $\family$ which satisfies the
conclusion of Theorem \ref{mth2} necessarily has the property that the number of points in $\family_n$ in \textit{any} disc of radius $1/\sqrt{n}$ is eventually uniformly bounded (with a bound depending only on $Q$ and $c$). Thus, such a family is necessarily a \textit{finite union} of asymptotically separated families, with a fixed upper bound on the number of them.
 This indicates that Theorem \ref{mth2} is a low-temperature phenomenon, i.e., that
a condition such as $\beta_n\to\infty$ is needed for the conclusion of the theorem to hold. Our method of proof uses a nontrivial adaptation of ``Landau's method'' of sampling and interpolating families, and is potentially useful for analyzing
more general point-processes which are ``lattice-like'' (i.e.~slight random perturbations of a deterministic lattice).

An estimate in a somewhat similar spirit as Theorem \ref{mth2} (i), formulated at deterministic zooming points in the bulk, is stated in \cite[Theorem 1]{ARSE}. This result, that depends on very different methods, applies to a more restrictive bulk regime where 
$\sqrt{n}\,\dist (p_n,\C\setminus S)\lr\infty$ sufficiently fast.
The fact that the densities in Theorem \ref{mth2} hold globally is of interest since the boundary regime is crucial with respect to freezing problems, see \cite{CSA} as well as the comments in Section \ref{CORE}.
\end{rem}

Returning to the issues, let us immediately dispose of part \ref{pa2} of Theorem \ref{mth2}, while simultaneously introducing certain concepts and results of central importance for our exposition.

Following \cite{A} we introduce a random variable, the \emph{distance from the droplet to the vacuum} by
$$D_n(\family_n)=\max_{1\le j\le n}\left\{\,\delta(\zeta_j)\,\right\},$$
where
$$\delta(\zeta)=\dist\left(\zeta,S\right).$$

The recent ``localization theorem'' in \cite[Theorem 2]{A} implies that under \eqref{freezing}, there is a constant $M=M(c)$ such that almost every random sample $\family=(\family_n)$
has the property that $\family_n\subset S_M=S_{M,n}$ for all large $n$, where $S_M$ is the \emph{$M$-vicinity of the droplet},
\begin{equation}\label{vicinity}S_M=S+D(0,M/\sqrt{n})=\left\{\, \zeta\, ;\, \delta(\zeta)<\frac M {\sqrt{n}}\,\right\}.\end{equation}

To spell it out explicitly: we have the convergence
\begin{equation}\label{local0}\lim_{n_0\lr\infty}\sum_{n=n_0}^\infty \bfP_n^{\,\beta_n}\left(\left\{D_n\ge \frac M {\sqrt{n}}\right\}\right)=0.\end{equation}

Using this, part \ref{pa2} of Theorem \ref{mth2} follows immediately.
Thus there remains only to prove parts \ref{pa1} and \ref{pa3}. This is done in the succeeding sections.

In fact, we shall deduce the following somewhat sharper statement, which clearly implies parts \ref{pa1} and \ref{pa3} of Theorem \ref{mth2}.

\begin{prop}\label{propop} (\emph{``Discrepancy estimates''})
Under the hypothesis of Theorem \ref{mth2}, assume that $L\ge 2$.

In the bulk case \ref{pa1} there exists a deterministic constant $C=C(c)$ such that, almost surely,
\begin{align}\label{eq_d1}
\limsup_{n\lr\infty}
\left|\, \# D(p_n,L/\sqrt{n}) - \Delta Q(p_*)\, L^{\,2}\, \right| \leq C L^{\,\alpha}, \qquad (\alpha=\tfrac 5 3).
\end{align}

In the boundary case \ref{pa3} there exists a deterministic constant $C=C(c)$ such that, almost surely,
\begin{align}\label{eq_d2}
\limsup_{n\lr\infty}
\left|\, \# D(p_n,L/\sqrt{n}) - \tfrac 1 2 \Delta Q(p_*)\, L^{\,2}\, \right| \leq C L^{\,\alpha} \log L, \qquad (\alpha=\tfrac 5 3).
\end{align}
\end{prop}

\begin{rem} Similar as for Theorem \ref{mth2}, our main point is that the above discrepancy estimate holds at any (family-dependent) sequence $\bfp=(p_n)_1^\infty$.
 Sharper discrepancy estimates for fixed observation disks that remain sufficiently far away from the boundary of $S$ have appeared in \cite{RS} in the setting of Fekete points. Also \cite[Theorem 1(2)]{ARSE} gives an estimate in this direction for $\beta$-ensembles, in the bulk. On the other hand, we expect the conclusion of Proposition \ref{propop} to be false if $\beta$ remains fixed independently of the number of particles $n$. (Related questions about fluctuations in fixed observation discs have also been studied, for example in \cite{FL}.)
\end{rem}

The exact value of the constant $\alpha$ in \eqref{eq_d1}, \eqref{eq_d2} is not important. Any other value $\alpha<2$ would have done as well for our purposes,
but the choice $\alpha=\tfrac 5 3$ turns out to lead to a particularly smooth and simple exposition.
(We do not make any claims about the
optimal value of $\alpha$ here; see Section \ref{CORE}.)

Our two main results on uniform separation and equidistribution reflect different aspects of the strong repulsions within the
system $\{\zeta_j\}_1^n$ which hold at low temperatures. Indeed, it is easy to see that a family may be equidistributed without being uniformly separated and vice versa. Moreover, it is a household fact  that
Landau's method for proving equidistribution of a family requires only a weak form of separation, namely that one can decompose it as a finite union
of smaller, uniformly separated families.

\begin{conj} We believe that Theorem \ref{mth} is sharp in the sense that if almost sure uniform separation holds for some sequence of inverse temperatures $\beta_n$,
then necessarily $\beta_n\ge c\log n$ for some constant $c>0$.
\end{conj}

Further comparison with other related work is found in Section \ref{CORE}.

\subsection{Plan of this paper}

In Section \ref{sec_1}, we introduce the basic objects of our theory, namely weighted polynomials. We also prove a few basic (pointwise-$L^p$ and gradient) estimates
for weighted polynomials.

In Section \ref{sec_2}, we prove Theorem \ref{mth} on uniform separation.

In Section \ref{FPD}, we recall a few facts pertaining to asymptotic properties of the reproducing kernel in spaces of weighted polynomials equipped with the $L^2$-norm. This kind
of asymptotic is needed for our later implementation of Landau's method.

In Section \ref{sec_samp}, we formulate suitable sampling and interpolation inequalities and prove that a random sample in the low temperature regime satisfies these
inequalities almost surely.

In Section \ref{sec_equ}, we prove the equidistribution theorem (Theorem \ref{mth2}) and the discrepancy estimates in Proposition \ref{propop}.

In Section \ref{CORE}, we compare with other related work and provide some concluding remarks.

\subsection*{Notational conventions}

For non-negative functions $f,g$, we write $f \lesssim g$ if there exists a constant $C>0$ such that $f \leq C g$ at every point. If $f \lesssim g$ and $g \lesssim f$, we write $f \asymp g$.
Generic constants are denoted $C$, $C_1$, etc. Their meaning may change from line to line.

The usual complex derivatives are denoted $\d=\tfrac 1 2 (\d_x-i\d_y)$ and $\dbar=\tfrac 1 2(\d_x+i\d_y)$. The symbol $\Delta=\d\dbar$ denotes $1/4$ of the standard Laplacian on $\C$.

An unspecified integral $\int f$ always means $\int f\, dA$ except when otherwise is indicated, where $dA=\tfrac 1 \pi \, dx\, dy$.
When $f$ is an integrable function on a disc $D_R$ of radius $R$, we will systematically denote its average value by
\begin{equation}\label{average}\fint_{D_R}f:=\frac 1 {R^2}\int_{D_R}f.\end{equation}

\subsection*{Acknowledgements} Y.~A.~wants to thank Seong-Mi Seo for useful discussions, and Simon Halvdansson \cite{H} for allowing us to reproduce Figure \ref{fekete}.

\section{Weighted polynomials and their basic properties}\label{sec_1}
In this section, we introduce the fundamental objects of this exposition, namely weighted polynomials. These are analogous to bandlimited functions in Landau's setting \cite{L}, and
they play a fundamental role in (for example) weighted potential theory \cite{ST}.

We now come to the definition.
Given any admissible potential $Q$, an integer $n$, and a holomorphic polynomial $q$ of degree at most $n-1$, we designate by
$$f(\zeta)=q(\zeta)\cdot e^{\,-nQ(\zeta)/2}$$
a \emph{weighted polynomial} of order $n$. The totality of such weighted polynomials will be denoted by the symbol $\calW_n$.

In random matrix theories (i.e., when $\beta=1$) it is customary to equip $\calW_n$ with the norm of $L^2=L^2(\C, dA)$. This is natural,
since the reproducing kernel of $(\calW_n,\|\cdot\|_{L^2})$ plays the role of a correlation kernel in this case.
When studying $\beta$-ensembles, it turns out to often be more natural to equip $\calW_n$ with
the norm in $L^{2\beta}$ (cf.~\cite{A,A2,CMMO}).
In the present work, we shall exploit both of these possibilities.

It is convenient to begin by proving a few frequently used estimates for weighted polynomials, starting with the following pointwise-$L^p$ estimate.

\begin{lem} \label{weh} Fix numbers $p>0$ and $s>0$ and suppose that $Q$ is $C^2$-smooth in a neighbourhood $U$ of a point $\zeta_0\in\C$. Let $f$ be a function of the form $f=u\cdot e^{\,-nQ/2}$
where $u$ is holomorphic in $U$ and suppose also that $\Delta Q\le M$ throughout $U$.
Then
for all $n$ large enough that $D(\zeta_0,s/\sqrt{n})\subset U$ we have
$$|f(\zeta_0)|^{\,p}\le n\cdot \frac {C^{\,p}}{s^{\,2}}\int_{D(\zeta_0,s/\sqrt{n})}|f|^{\,p},\qquad (C=e^{\,Ms^{\,2}/2}).$$
\end{lem}

\begin{proof} Consider the function
$$F(\zeta)=|f(\zeta)|^{\,p}\cdot e^{\,Mpn|\zeta-\zeta_0|^{\,2}/2}.$$
By hypothesis, we have whenever $D(\zeta_0,s/\sqrt{n})\subset U$ that
$$\Delta\log F\ge -p n\Delta Q/2+Mnp /2\ge 0\quad \text{on}\quad D(\zeta_0,s/\sqrt{n}).$$
This makes $F$ (logarithmically) subharmonic on $D(\zeta_0,s/\sqrt{n})$, so
$$F(\zeta_0)\le \frac n{s^{\,2}}\int_{D(\zeta_0,s/\sqrt{n})}F.$$
In turn, this implies
$$|f(\zeta_0)|^{\,p}\le n \cdot \frac{e^{\,Mp s^{\,2}/2}}{s^{\,2}}\int_{D(\zeta_0,s/\sqrt{n})}|f|^{\,p}.$$
\end{proof}

In the following, given a subset $\Omega\subset \C$ and a constant $M>0$ we write $\Omega_M$ for the $M/\sqrt{n}$-neighbourhood of $\Omega$, i.e.,
$$\Omega_M=\Omega+D(0,M/\sqrt{n}).$$

\begin{cor} \label{v2} Let $\Lambda$ be a neighbourhood of the droplet $S$ and assume that $Q$ is $C^2$-smooth in a neighbourhood $U$ of $\overline{\Lambda}$
with $\Delta Q\le M$ there. Then for each subset $\Omega$ of $\Lambda$ and each $2s_0$-separated configuration $\family_n=\{\zeta_j\}_1^n$ contained in $\Omega$ we have
$$\frac 1 n\sum_{j=1}^n|f(\zeta_{j})|^{\,p}\le \frac {C^{\,p}}{s_0^{\,2}}\int_{\Omega_{s_0}}|f|^{\,p},\qquad \forall f\in\calW_{n},\, \forall p>0,\forall n\ge n_0,$$
where $C$ depends only on $M$ and $s_0$, and $n_0$ is chosen with $s_0/\sqrt{n_0}<\dist(\Lambda,\C\setminus U)$.
\end{cor}

\begin{proof} For each $j$, by Lemma \ref{weh},
$$|f(\zeta_{j})|^{\,p}\le n\cdot \frac {C^{\,p}}{s_0^{\,2}}\int_{D(\zeta_{j},s_0/\sqrt{n})}|f|^{\,p},\qquad \forall f\in\calW_n.$$
Here the discs $D(\zeta_{j},s_0/\sqrt{n})$ are pairwise disjoint for $j=1,\ldots,n$ by virtue of the $2s_0$-separation. Hence summing in $j$
proves the desired statement.
\end{proof}

Another basic tool is provided by the following ``Bernstein type'' estimate, cf. \cite{AOC,A2,MMOC}. (We remind that ``$\fint$'' denotes the average value, cf. \eqref{average}.)

\begin{lem}\label{grdlem} Let $K$ be a compact set such that $Q$ is $C^2$-smooth in a neighborhood $U$ of $K$.
Pick an integer $n_0$ so that $1/\sqrt{n_0}< \dist (K,\C\setminus U)$.
Also let $f\in\calW_n$ and $p\in K$ be such that $f(p)\ne 0$. Then there is a constant $C$ depending only on $\max_K\{|\Delta Q|\}$ such that for all $n\ge n_0$,
$$|\nabla|f|(p)|\le C\sqrt{n}\fint_{D(p,1/\sqrt{n})}|f|.$$
\end{lem}

\begin{proof}
Fix an integer $n \ge n_0$ and a point $p\in K$ and define a holomorphic polynomial $H_p$ by
$$H_p(\zeta)=Q(p)+2\d Q(p)\cdot(\zeta-p)+\d^{\,2} Q(p)\cdot (\zeta-p)^{\,2}.$$
A Taylor expansion about $\zeta=p$ gives (for $n\ge n_0$)
\begin{equation}\label{6}n\cdot |Q(\zeta)-\re H_p(\zeta)|\le n\cdot |\Delta Q(p)|\cdot |\zeta-p|^{\,2}+n\cdot o(|\zeta-p|^{\,2})\le M,\quad \text{when}\quad |\zeta-p|\le 1/\sqrt{n}.
\end{equation}
The constant $M$ can be chosen independently of the particular point $p\in K$ and of $n\ge n_0$ by choosing a suitable $M$ strictly larger than the maximum of $|\Delta Q|$
over $K$.

Now if $f=q\cdot e^{\,-nQ/2}\in\calW_n$  satisfies $f(p)\ne 0$ then by a straightforward computation,
$$|\nabla|f|(p)|=|q'(p)-n\d Q(p)\cdot q(p)|e^{\,-nQ(p)/2}.$$
(And moreover, $|\nabla(|q|e^{\,-nQ/2})|=2|\d(q^{1/2}\bar{q}^{1/2}e^{\,-nQ/2})|=|q'-qn\d Q|e^{\,-nQ/2}$.)

In a similar way we find that
$$|\nabla(|q|e^{\,-n\re H_p/2})(\zeta)|=\left|\frac d {d\zeta}(qe^{\,-nH_p/2})(\zeta)\right|.$$
Inserting $\zeta=p$ we see that
$$|\nabla|f|(p)|=\left|\frac d {d\zeta}(qe^{\,-nH_p/2})(p)\right|.$$

Using a Cauchy estimate, we now find that for each $r$ with $1/(2\sqrt{n})\le r\le 1/\sqrt{n}$,
\begin{align*}\left|\frac d {d\zeta}(qe^{\,-nH_p/2})(p)\right|&=\frac 1 {2\pi}\left|\int_{|\zeta-p|=r}\frac {q(\zeta)e^{\,-nH_p(\zeta)/2}}{(\zeta-p)^2}\, d\zeta\right|\\
&\le \frac {2n}{\pi}\int_{|\zeta-p|=r}|q|e^{\,-n\re H_p/2}\, |d\zeta|.
\end{align*}

In view of \eqref{6}, the last expression is dominated by
$$\frac {2n}\pi e^{M/2}\int_{|\zeta-p|=r}|f(\zeta)|\,|d\zeta|.$$
Integrating in $r$ over $1/(2\sqrt{n})\le r\le 1/\sqrt{n}$, we find that
\begin{align*}|\nabla|f|(p)|&\le \frac {4n^{\,3/2}}\pi e^{\,M/2}\int_{1/(2\sqrt{n})}^{1/\sqrt{n}}\, dr\, \int_{|\zeta-p|=r}|f(\zeta)|\,|d\zeta|\\
&\le
4e^{\,M/2}\sqrt{n}\fint_{D(p,1/\sqrt{n})}|f|\, dA.\end{align*}

The proof of the lemma is complete.
\end{proof}

Finally, we recall a (rather weak, but sufficient for our purposes) version of the maximum principle of weighted potential theory. See \cite[Theorem 2.1]{ST} or \cite[Lemma 2.3]{A} for proofs and
more refined versions.

\begin{lem}\label{lem_maxp}
Let $Q$ be an admissible potential. Then each weighted polynomial $f\in\calW_n$ assumes a global maximum on the droplet $S$, i.e., $\|\,f\,\|_{L^\infty(\C)}=\|\,f\,\|_{L^\infty(S)}$.	
\end{lem}

\section{Uniform separation}\label{sec_2}
In this section we prove that low-temperature Coulomb gas ensembles are almost surely uniformly separated (Theorem \ref{mth}). To this end we fix a potential $Q$ satisfying the requirements in Theorem
\ref{mth}. More specifically, we fix a neighbourhood $\Lambda$ of $S$ such that $Q$ is $C^2$-smooth in a neighbourhood of the closure $\overline{\Lambda}$.

At the outset, the inverse temperature $\beta=\beta_n$ can be taken arbitrarily subject only to a mild constraint such as $\beta\ge \beta_0$ for some fixed constant $\beta_0>0$. The much more
stringent condition $\beta_n\gtrsim\log n$ will come into play later in our proof.

We now pick a configuration $\{\zeta_k\}_1^n$ and associate the
\textit{weighted Lagrange polynomial} $\ell_j\in\calW_n$ by
\begin{align}\label{nota}
\ell_j(\zeta)=l_j(\zeta)\cdot e^{\,-n(Q(\zeta)-Q(\zetaj))/2},\qquad
 l_j(\zeta)=\prod_{k\ne j}
\frac{\zeta-\zetak}{\zetaj-\zetak}.
\end{align}
Note that $\ell_j(\zetak)=\delta_{jk}$.

We shall in the following regard $\ell_j(\zeta)$ as a random function, depending on the random sample $\family_n=\{\zeta_k\}_1^n$ from $\bfP_n^{\,\beta}$. These kinds of
random functions were systematically used in the papers \cite{A,A2}, and we shall here continue in this direction. (Somewhat related ``Lagrange sections'' have been used previously
in a context of complex geometry, \cite{CMMO}.)

We now introduce random functions of the form
\begin{equation}Y_{j,f}(\zeta):=f(\zeta,\zetaj)\cdot |\ell_j(\zeta)|^{\,2\beta},\end{equation}
where $f$ is an arbitrary but fixed complex-valued, measurable function on $\C^2$, integrable with respect to the normalized Lebesgue measure $dA_2$.

Finally, we let $\mu_1$, the ``1-point measure''
be the distribution of the random variable $\zeta_j$, i.e., $\mu_1$ is the probability measure on $\C$ such that
\begin{equation}\mu_1(D):=\bfP_n^{\,\beta}\left(\left\{\,\zeta_j\in D\,\right\}\right)\end{equation}
for Borel sets $D$. Of course $\mu_1$ is independent of the particular choice of $j$, $1\le j\le n$. (Indeed,
the Radon-Nikodym derivative $d\mu_1/dA$ equals to $\tfrac 1 n \bfR_n^{\, \beta}$.)

The following basic lemma generalizes \cite[Lemma 2.5]{A}.

\begin{lem} \label{exid} The following exact identity holds,
$$\bfE_n^\beta\int_\C Y_{j,f}(\zeta)\, dA(\zeta)=\int_{\C^2}f(\zeta,\eta)\, (d\mu_1\otimes dA)(\zeta,\eta).$$
\end{lem}

\begin{proof} We can without loss of generality take $j=1$.

It is easy to check (as in \cite{A,A2}) that
$$|\ell_1(\zeta)|^{\,2\beta}e^{\,-\beta H_n(\zeta_1,\zeta_2,\ldots,\zeta_n)}=e^{\,-\beta H_n(\zeta,\zeta_2,\ldots,\zeta_n)}.$$
Consequently,
\begin{align*}\bfE_n^{\,\beta_n}\int_\C Y_{1,f}(\zeta)\, dA(\zeta)&=\int_{\C^{n+1}}f(\zeta,\zeta_1)\frac 1 {Z_n^{\,\beta}}|\ell_1(\zeta)|^{\,2\beta}e^{\,-\beta H_n(\zeta_1,\zeta_2,\ldots,\zeta_n)}\,
dA_{n+1}(\zeta,\zeta_1,\ldots,\zeta_n)\\
&=\int_{\C^{n+1}}f(\zeta,\zeta_1)\frac 1 {Z_n^{\,\beta}}e^{\,-\beta H_n(\zeta,\zeta_2,\ldots,\zeta_n)}\, dA_{n+1}(\zeta,\zeta_1,\ldots,\zeta_n)\\
&=\int_{\C^2}f(\zeta,\zeta_1)\, (d\mu_1\otimes dA)(\zeta,\zeta_1).
\end{align*}
\end{proof}

Recall that we have fixed a neighbourhood $\Lambda$ of the droplet in which $Q$ is $C^2$-smooth and strictly subharmonic. As shown in \cite[Theorem 3]{A}, the system $\{\zeta_j\}_1^n$ is contained in $\Lambda$ with very large probability:
\begin{equation}\label{ex0}\bfP_n^{\,\beta_n}\left(\left\{\,\{\zeta_j\}_1^n\not \subset \Lambda\,\right\}\right)\le Ce^{\,-c_1\beta_n n}\end{equation}
where $C$ and $c_1$ are positive constants.

We next claim that if $C$ is any constant with $C>1$ then there is $n_0$ large enough so that when $n\ge n_0$, we have the following estimate for the conditional expectation of the average value
of $|\ell_1|^{\,2\beta}$ over the disc $D(\zeta_1,2/\sqrt{n})$,
\begin{equation}\label{enoch}\bfE_n^{\,\beta}\left[\fint_{D(\zeta_1,2/\sqrt{n})}|\ell_1|^{\,2\beta}\, dA\,\bigm|\,\{\zeta_j\}_1^n\subset\Lambda\right]\le C.\end{equation}

To prove this we put $\eps_1=2/\sqrt{n}$ and apply Lemma \ref{exid} with $f$ the indicator function of the set $E(\lambda,\eps_1)$ defined by
$$E(\Lambda,\eps_1)=\{(\zeta,\zeta_1)\, ;\, \zeta_1\in \Lambda\,\;\, |\zeta-\zeta_1|<\eps_1\}.$$
From Lemma \ref{exid} we have
\begin{equation*}\bfE_n^{\,\beta}\int_\C Y_{1,\1_{E(\Lambda,\eps_1)}}(\zeta)\, dA(\zeta)=\int_\Lambda \mu_1(D(\zeta_1,\eps_1))\, dA(\zeta_1).\end{equation*}

Here the right hand side is simplified by writing $\mu_1(D(\zeta,\eps_1))$ as the convolution $\mu_1*\1_{D(0,\eps_1)}(\zeta)$, which gives
$\int_\C \mu_1(D(\zeta_1,\eps_1))\, dA(\zeta_1)=
\eps_1^{\,2}$.
In view of \eqref{ex0}, we obtain \eqref{enoch}. (We also see that $C$ in \eqref{enoch} can be chosen arbitrarily close to $1$ by choosing $n_0$ large enough.)

In the following we assume that $n_0$ is large enough that $\Lambda+D(0,3/\sqrt{n_0})$ is contained in the set where $Q$ is $C^2$-smooth, and we take $n\ge n_0$.

Now fix a large parameter $\lambda$. By Chebyshev's inequality and \eqref{enoch} (or rather its counterpart with $\zeta_1$ replaced by $\zeta_j$)
$$\bfP_n^{\,\beta}\left(\left\{\,\fint_{D(\zeta_j,2/\sqrt{n})}|\ell_j(\zeta)|^{\,2\beta}\, dA(\zeta)>\lambda\,\right\}\, \bigm|\,\{\zeta_k\}_1^n\subset \Lambda\,\right)\le C \frac 1 \lambda.$$
A union bound thus gives
\begin{equation}\label{nub}\bfP_n^{\,\beta}\left(\left\{\,\max_{1\le j\le n}\fint_{D(\zeta_j,2/\sqrt{n})}|\ell_j(\zeta)|^{\,2\beta}\, dA(\zeta)>\lambda\,\right\}\, \bigm|\,\{\zeta_k\}_1^n\subset \Lambda\,\right)\le
 C\frac n \lambda.
\end{equation}

Now assume that $\{\zeta_k\}_1^n\subset \Lambda$.

Assuming that $\beta\ge 1/2$, we find by Lemma \ref{grdlem} and Jensen's inequality that for all $\zeta\in\Lambda$ with $f(\zeta)\ne 0$,
$$|\nabla|f|(\zeta)|^{\,2\beta}\lesssim C^{\,2\beta}n^{\,\beta}\fint_{D(\zeta,1/\sqrt{n})}|f|^{\,2\beta}\, dA,\qquad (f\in\calW_n).$$
Hence using \eqref{nub}, we see that there is a constant $C$ independent of $\beta$ with $\beta\ge 1/2$ such that with probability at least $1-Cn/\lambda$ we have
\begin{equation}\label{holl}|\nabla|\ell_j|(\zeta)|\le Cn^{\,\frac 12}\lambda^{\,\frac 1{2\beta}}.\end{equation}
Here \eqref{holl} holds for all $\zeta$ in a neighbourhood of $\Lambda$, with the exception of the finite set of points at which $|\ell_j|$ is not differentiable
(namely the points $\zeta_k$ with $k\ne j$).

Now choose $j,k\in\{1,\ldots,n\}$ with $j\ne k$ so that the distance $|\zeta_j-\zeta_k|$ is minimal. Integrating \eqref{holl} along the straight line-segment $\gamma=[\zeta_j,\zeta_k]$
between these points, we find
\begin{equation}\label{lett}1=||\ell_j(\zeta_j)|-|\ell_j(\zeta_k)||=\left|\,\int_{\gamma}\nabla|\ell_j(\zeta)|\bigcdot (d\re\zeta,d\im\zeta)\,\right|\le C\sqrt{n}\cdot \lambda^{\, \frac1{2\beta}}
\cdot|\zeta_j-\zeta_k|.
\end{equation}
(We have here assumed that $Q$ is smooth in a neighbourhood of the segment $\gamma$. This may of course be assumed by somewhat shrinking $\Lambda$ if necessary.)

Now recall that the spacing of the sample $\{\zeta_l\}_{l=1}^n$ is
$$\sep_n=\sqrt{n}\cdot |\zeta_j-\zeta_k|,$$
so \eqref{lett} says that $\sep_n\ge \lambda^{\,-\frac 1{2\beta}}/C$.

We have shown that, with probability at least $1-Cn/\lambda-Ce^{\,-c_1\beta n}$, we have that $\sep_n\ge \lambda^{\,-1/2\beta}/C$.

Now pick $\eps>0$ and put $\lambda=n^{\,2+\eps}$. We then find with probability at least $1-C_1n^{\,-1-\eps}$ (for some new constant $C_1$) that
$$\sep_n\ge C_2e^{\,-\frac {(2+\eps)\log n}{2\beta}}\ge C_2e^{\,-\frac {2+\eps}{2c}},\qquad \text{if}\qquad \beta \ge c\log n.$$

It follows that if we pick a family $\family=(\family_n)_n$ in the low-temperature regime $\beta_n\ge c\log n$ where $c>0$, and if we take
\begin{equation}\label{sep0eq}s_0=C_2e^{\,-\frac {2+\eps}{2c}}\end{equation}
then
$$\sum_{n=n_0}^\infty \bfP_n^{\,\beta_n}\left(\left\{\,\sep_n(\family_n)<s_0\,\right\}\right)\le C_1\sum_{n=n_0}^\infty \frac 1 {n^{\,1+\eps}}.$$
The right hand side clearly tends to $0$ as $n_0\lr \infty$.

We have shown that almost every family is $s_0$-separated, and our proof of Theorem \ref{mth} is complete. q.e.d.

\medskip

\begin{rem} Our proof above shows that there are positive constants $a$ and $b$ such that for all $\beta\ge 1/2$ and all $\lambda>0$,
\begin{equation}\label{spaz}\bfP_n^{\,\beta}(\sep_n\ge a\lambda^{-\frac 1 {2\beta}})\ge 1-b\tfrac n\lambda.\end{equation}
For example, choosing $\lambda=\lambda_n= nm_n$ where $m_n\to\infty$ slowly, say $m_n=\log n$, we obtain the result that for fixed $\beta$, we have $\sep_n\ge (n\log n)^{-\frac 1 {2\beta}}$ with large probability if $n$ is large enough.
\end{rem}

\section{Further preliminaries: The determinantal case}\label{FPD}

In this section, we recall a few facts pertinent to the well-studied determinantal case $\beta=1$. Perhaps surprisingly, asymptotic properties of the $\beta=1$ kernel are crucially used in our subsequent analysis of low
temperature ensembles, when $\beta_n\gtrsim \log n$.

Consider the Coulomb gas $\{\zeta_j\}_1^n$ in external field $nQ$ at inverse temperature $\beta=1$. This is a determinantal process, i.e., the $k$-point intensity function $\bfR_{n,k}$
is given by a determinant: $$\bfR_{n,k}(\eta_1,\ldots,\eta_k)=\det(\bfK_n(\eta_i,\eta_j))_{i,j=1}^k,$$
where $\bfK_n(\zeta,\eta)$ is a suitable ``correlation kernel''.

In fact, as is well-known, $\bfK_n$ may be taken as the reproducing kernel of the space
$\calW_n$ of weighted polynomials, regarded as a subspace of $L^2=L^2(\C,dA)$. (Cf.~e.g.~\cite{F,M}.)
In the following, we shall always let $\bfK_n$ denote this \emph{canonical correlation kernel}.

Similar as in the earlier works \cite{A0,AOC}, we shall discuss asymptotic properties for the one-point function
$$\bfR_n(\zeta)=\bfK_n(\zeta,\zeta)$$ as well as some off-diagonal estimates pertaining to the
\textit{Berezin kernel}
$$\bfB_n(\zeta,\eta)=\frac {|\bfK_n(\zeta,\eta)|^{\,2}}{\bfK_n(\zeta,\zeta)}.$$

\subsection{Scaling limits and the lower bound property}
The following result will be crucial for our subsequent
usage of sampling and interpolation inequalities. As always we use the symbol $S_M$ to denote the $M/\sqrt{n}$-neighbourhood of the droplet $S$,
$$S_M=S+D(0,M/\sqrt{n}).$$

\begin{thm} \label{lubb} Let $Q$ be an external potential.
\begin{enumerate}[label=(\alph*)]
\item \label{trivial} If $Q$ is $C^2$-smooth in a neighbourhood of the droplet, then there exists a constant $C$ such that
$$\sup_{\zeta\in\C}\bfR_n(\zeta)\le C\cdot n.$$
\item \label{harder} If $Q$ is real-analytic and strictly subharmonic in a neighbourhood of $S$, and if $\d S$ is everywhere regular, then for any
$M\ge 0$ there is a constant $c_M>0$ such that
$$\inf_{\zeta\in S_M}\bfR_n(\zeta)\ge c_M\cdot n.$$
\end{enumerate}
\end{thm}

Let us briefly recall the proof of \ref{trivial}, which follows from the identity (a general property of reproducing kernels \cite{DuS})
\begin{equation}\label{dus}\bfK_n(\zeta,\zeta)=\sup\left\{\,|f(\zeta)|^2\, ;\,f\in\calW_n,\,\|\,f\,\|\le 1\,\right\}.\end{equation}

Recalling that $|f(\zeta)|^2\le Cn\|\,f\,\|^{\,2}$ for each $f\in\calW_n$ by Lemma \ref{weh}, we finish the proof of \ref{trivial}.

Our proof of the ``lower bound property'' \ref{harder} is more subtle and
requires some preparation.

\begin{rem}
The validity of a property closely related to \ref{harder} was taken as an assumption in \cite{A0}
(``universally translation-invariant property''), while here (b) is shown to be a consequence of the general assumptions on $Q$.
\end{rem}

Our proof of \ref{harder} uses an apriori knowledge of all possible subsequential scaling limits, which will be of frequent use in the sequel. To define
these limits, we recall the standard procedure for taking microscopic limits in planar Coulomb gas ensembles.

Given any sequence $(p_n)$ ($p_n\in S_M$) we consider the \emph{magnification} (or \emph{blow-up}) about $p_n$ by which we mean the mapping
\begin{equation}\label{rescaling}\magn_n:\zeta\longmapsto z,\qquad z=\magn_n(\zeta)=\sqrt{n\Delta Q(p_n)}\cdot (\zeta-p_n)\cdot e^{\,-i\theta_n}.\end{equation}
Here the angular parameter $\theta_n\in\R$ can be chosen arbitrarily, according to convenience. (Note that by assumption \ref{Q6} we have the estimate
$\Delta Q(p_n)\ge c_1$ for some positive $c_1$ depending only on $Q$.)

The rescaled system $\{z_j\}_1^n$ where $z_j=\magn_n(\zeta_j)$ is a new determinantal process with correlation kernel
\begin{equation}\label{recc}K_n(z,w)=\frac 1 {n\Delta Q(p_n)}\bfK_n(\zeta,\eta),\qquad z=\magn_n(\zeta),\quad w=\magn_n(\eta).\end{equation}

Following the convention in \cite{AKM}, we denote by italic symbols objects pertaining to the rescaled process. In particular, we write
$$R_n(z)=K_n(z,z),\qquad B_n(z,w)=\frac {|K_n(z,w)|^{\,2}}{K_n(z,z)}$$
for the 1-point function and the Berezin kernel rooted at $z$, respectively.

We shall use throughout the symbol $G$ for the
usual \textit{Ginibre kernel}
$$G(z,w)=e^{\,z\bar{w}-|z|^{\,2}/2-|w|^{\,2}/2},$$
and we say that a function $L(z,w)$ is ``Hermitian-entire'' if it is Hermitian (i.e. $\overline{L(z,w)}=L(w,z)$) and entire as a function of $z$ and of $\bar{w}$.

We remind that a Hermitian function $c$ is called a \textit{cocycle} if it takes
the form $c(\zeta,\eta)=g(\zeta)\overline{g(\eta)}$ where $g$ is continuous and unimodular. It is a basic fact of determinantal point-processes that a correlation kernel
is only determined ``up to cocycle'', namely if $K$ is
a correlation kernel then $cK$ is another one.

The following lemma follows from \cite[Lemma 2]{AKMW}.

\begin{lem} \label{lemlim} Suppose that $Q$ is real-analytic and strictly subharmonic in a neighbourhood of the closure of a subset $\Omega\subset\C$ and suppose $p_n\in\Omega$. Then there exists a sequence of cocycles $c_n$ so that each subsequence
of the rescaled kernels $(c_nK_n)_1^\infty$ has a further subsequence converging locally uniformly on $\C^2$ to a limiting kernel $K$ of the form
$$K(z,w)=G(z,w)\cdot L(z,w),$$
where $L$ is some Hermitian entire function called a ``holomorphic kernel''.
\end{lem}

\begin{proof}[Remark on the proof] (Cf.~\cite[Lemma 2]{AKMW}.) The existence of suitable limiting kernels is shown using a standard normal families argument in \cite{AKM}. We note that the real-analyticity of $Q$ in
$\Omega$ is crucially used in this argument, which otherwise works exactly in the same way irrespective of whether the point $p_n$ is fixed or $n$-dependent, and whether the angle parameter $\theta_n$
is $n$-dependent or
not.
\end{proof}

A subsequential limit $K=GL$ in Lemma \ref{lemlim} is the correlation kernel of a unique limiting determinantal point field $\{z_j\}_1^\infty$ (see e.g.~\cite{S}).
This limit in turn is determined by the \textit{limiting $1$-point function}
$$R(z)=\lim R_{n_k}(z)=L(z,z).$$ For example, if $(p_n)$ is in the bulk regime we obtain
that $R\equiv 1$, which is characteristic for the usual infinite Ginibre ensemble (with correlation kernel $G$). This universal bulk-limit follows from well-known
(``H\"{o}rmander-type'') estimates, e.g. \cite[Theorem 5.4]{AKM}, and in particular is independent of choice of angle-parameters $\theta_n$ in \eqref{rescaling}.

If the zooming points $(p_n)$ are in the boundary regime, the microscopic behaviour can be described in terms of the $\erfc$-kernel $L(z,w)=F(z+\bar{w})$ where $F$ is the function
\begin{align}\label{eq_fbf}
\fbf(z)
=\frac{1}{ \sqrt{2\pi}} \int_{-\infty}^{\,0} e^{-(z-t)^{\,2}/2}\, dt=\frac{1}{2} \erfc\left( \frac z {\sqrt{2}}\right).
\end{align}

We now fix the angle-parameters $\theta_n$ appropriately. For this, assume that the boundary $\d S$ consists of finitely many $C^1$-smooth
Jordan curves.

Consider for each (large) $n$ the unique point $q_n\in \d S$ that is closest to $p_n$. We choose $\theta_n$ so that
$e^{\,i\theta_n}$ is the outwards unit normal to $\d S$ at $q_n$.

We can further assume (by passing to a subsequence if necessary)
that the limit
\begin{equation}\label{limo}l=\lim_{n\lr\infty}\sqrt{n\Delta Q(p_n)}\cdot (p_n-q_n)\cdot e^{\,-i\theta_n}\end{equation}
exists.

\begin{figure}[h]
\begin{center}
\includegraphics[width=.3\textwidth]{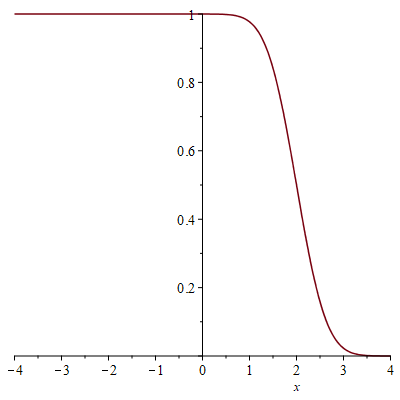}\hfil
\includegraphics[width=.3\textwidth]{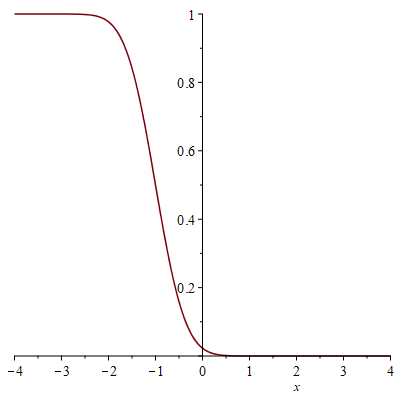}
\end{center}
\caption{The density profiles $x\longmapsto F(2x+2l)$ with $l=-2$ (left) and $l=1$ (right).}
\label{fig0}
\end{figure}

\begin{thm} \label{prop_univ} Suppose that $Q$ is real analytic and strictly subharmonic in a neighbourhood of the droplet.
\begin{enumerate}[label=(\Alph*)]
\item \label{aa} If $(p_n)$ is in the bulk regime there is a unique limiting $1$-point function, namely $R\equiv 1$.
\item \label{bb} Suppose that $S$ is connected and the boundary $\d S$ is everywhere smooth.  Then if $(p_n)$ is in the boundary regime and the limit \eqref{limo} holds,
there is also a unique limiting $1$-point function, namely $$R(z)=F(z+\bar{z}+2l).$$
\end{enumerate}
\end{thm}

In the language of point-processes, the theorem says that the $n$-point process $\{z_j\}_1^n$ converges to the point field with correlation kernel $K=G$ in case \ref{aa} and $K=K_l$
in case \ref{bb} where
\begin{equation}\label{Kl}K_l(z,w)=G(z,w)F(z+\bar{w}+2l).\end{equation}
These kernels interpolate between the Ginibre kernel $G$ at $l=-\infty$ and the trivial kernel $0$ at $l=+\infty$.
Figure \ref{fig0} shows the corresponding density profile $R(x)=K_l(x,x)$ for a few specific values of $l$.

\begin{proof} Part \ref{aa} has already been proved above and \ref{bb} follows immediately from the leading term of the edge-asymptotic theorem in \cite{HW2} (or rather, by its counterpart
for $n$-dependent zooming points $p_n$). See also \cite{HW}, which concerns simply-connected droplets (while
\cite{HW2} applies to multi-connected ones).
\end{proof}

\begin{proof}[Proof of Theorem \ref{lubb}\ref{harder}] If the conclusion of \ref{harder} fails,
there is a sequence $(p_{n})$ with $p_n\in S_M$ for each $n$, and a subsequence $n_k$ of the positive integers such that $\lim_{k\to\infty}\tfrac 1 {n_k}\bfR_{n_k}(p_{n_k})= 0$.
Rescaling about $p_n$ as above, we then find that
$\lim_{k\to\infty}R_{n_k}(0)= 0$. This contradicts Theorem \ref{prop_univ}.
\end{proof}

\subsection{Some auxiliary estimates}

A frequently useful property of the Ginibre kernel is its Gaussian off-diagonal decay,
\begin{equation}\label{Gdecay}|G(z,w)|^{\,2}=e^{\,-|z-w|^{\,2}}.\end{equation}
For the kernels $K_l$ in \eqref{Kl} there is also off-diagonal decay, albeit much slower.

\begin{lem} \label{kboo} There is a constant $C$ independent of $l\in\R$ and $z,w\in\C$ such that
$$|K_l(z,w)|\le C\frac {e^{\,-|\re(z-w)|^{\,2}/2}} {1+|\im(z-w)|}.$$
\end{lem}

\begin{proof} The lemma follows from the proof of \cite[Lemma 8.5]{AOC}; it is convenient to recall the argument in some detail.

We start with the observation that
$$|K_l(z,w)|^{\,2}=e^{\,-|z-w|^{\,2}}|F(z+\bar{w}+2l)|^{\,2}.$$
Now write
$z+\bar{w}+2l=a+ib$ where $a=\re(z+\bar{w}+2l)$ and $b=\im(z-w)$.

By Cauchy's theorem, we have $F(z)=\frac 1 {\sqrt{2\pi}}\int_{-\infty}^z e^{\,-u^2/2}\, du$ where we may choose any suitable contour of integration connecting $-\infty$
to the point $z$. We choose the contour $\gamma=(-\infty,a]\cup [a,a+ib]$ and find
\begin{align*}|F(z+\bar{w}+2l)|&\le \left|\,\frac 1 {\sqrt{2\pi}}\int_{-\infty}^{\,a}e^{\,-u^2/2}\, du\,\right|+\left|\,\frac 1 {\sqrt{2\pi}}\int_a^{\,a+ib}e^{\,-u^2/2}\, du\,\right|\\
&\le F(a)+\left|\,\frac {e^{-a^2/2}} {\sqrt{2\pi}}\int_0^be^{\,-iat+t^2/2}\, idt\,\right|\\
&\le 1+\frac {e^{\,-a^{\,2}/2+b^{\, 2}/2}} {\sqrt{2\pi}}\cdot D\left(\frac b{\sqrt{2}}\right),
\end{align*}
where $D(x)$ is \emph{Dawson's integral},
$$D(x)=e^{\,-x^{\,2}}\int_0^{\,x}e^{\,t^{\,2}}\, dt.$$

We have shown that
\begin{align}\label{p0}|K_l(z,w)|\le e^{\,-|z-w|^{\,2}/2}+\frac {e^{\,-[\re(z-w)]^{\,2}/2}}{\sqrt{2\pi}}\cdot D\left(\frac {|\im(z-w)|}{\sqrt{2}}\right).\end{align}

We next use the asymptotic
\begin{align}\label{p1}D(x)=\tfrac 1 {2x}+O(\tfrac 1 {x^{\,3}}),\qquad (x\lr\infty),\end{align}
see e.g.~\cite[p.~406]{SO}. Together with \eqref{p0} this implies
\begin{equation}\label{p2}\begin{split}|z-w|\cdot|K_l(z,w)|\le \,&|z-w|\cdot e^{\,-|z-w|^{\,2}/2}\\
 &+|\re(z-w)|\cdot e^{\,-[\re(z-w)]^{\, 2}/2}\frac 1 {\sqrt{2\pi}} \cdot D\left(\frac {|\im(z-w)|}{\sqrt{2}}\right)\\
&+
e^{\,-[\re(z-w)]^{\, 2}/2}\cdot |\im(z-w)|\cdot D\left(\frac {|\im(z-w)|}{\sqrt{2}}\right)\\
 &\le C_1\end{split}\end{equation}
with a large enough absolute constant $C_1$ (independent of $z,w$, and $l$).

Combining the estimates \eqref{p0}, \eqref{p1} and \eqref{p2}, we finish the proof.
\end{proof}

Recall that a set $E \subseteq \C$ is said to have finite perimeter if its indicator function $\1_E$ has bounded variation, and in this case we define
$$\perim E=\operatorname{var} \1_E$$
(see e.g.~\cite[Section 5]{evga92}). This is the same as the linear Hausdorff measure \cite{GM} of the measure theoretic boundary of $E$. In our subsequent applications, the set $E$
will have a piecewise smooth boundary, so the practical-minded reader may think of the usual arclength.
We will also denote by $|E|=\int_E\, dA$ the normalized Lebesgue measure of $E$.

We shall now estimate two integrals that come in naturally in connection with our method for proving equidistribution (cf.~also \cite{AOC,A0}). We will use the notation $\log^+t=\max\{0,\log t\}$.

\begin{lem}\label{lemma_perims}
Let $E$ be a bounded measurable subset of $\C$ with finite perimeter. There is then a universal constant $C$ such that
\begin{align}\label{eq_bv_1}
&\int_E \int_{\C \setminus E} |G(z,w)|^{\,2}\, dA(z)\,dA(w)
\leq C \cdot \perim E,
\\\label{eq_bv_2}
&\int_E \int_{\C \setminus E} |K_l(z,w)|^{\,2}\, dA(z)\,dA(w)
\leq C \cdot \perim E \cdot \big(1+
\log^+ \tfrac{|E|}{\perim E} \big).
\end{align}
\end{lem}
\begin{proof}
We shall use the following estimate for regularization with an integrable convolution kernel $\psi: \mathbb{C} \to \mathbb{R}$:
\begin{align}\label{eq_conv}
\|\, \1_E * \psi - (\smallint \psi) \cdot \1_E \,\|_{L^1(dA)}
\leq \perim E \cdot \int_{\C} |z| |\psi(z)|\, dA(z).
\end{align}
A proof of \eqref{eq_conv} under the assumption that $\smallint \psi =1$ can be found, e.g., in \cite[Lemma 3.2]{AGR}, while the general case follows by homogeneity.

To prove \eqref{eq_bv_1} we take
$\psi(z) := e^{\,-|z|^{\,2}}$
and $C_1 := \int |z| \psi(z)\, dA(z)$. Then \eqref{Gdecay} and \eqref{eq_conv} give
\begin{align*}
&\int_E \int_{\C \setminus E} |G(z,w)|^{\,2} \,dA(z)\, dA(w)
= \int_{\C \setminus E} \Big[\int_E \psi(z-w)\, dA(w) - (\smallint \psi) \cdot
\1_E(z)\Big]\,
dA(z)
\\
&\qquad\leq \|\, \1_E * \psi - (\smallint \psi) \cdot \1_E \,\|_{1} \leq C_1 \cdot \perim E.
\end{align*}

For \eqref{eq_bv_2}, we
select $R \geq 1$ and set
$$\psi(z):= \frac {e^{\,-|\re(z)|^{\,2}}} {(1+|\im(z)|)^{\,2}},\qquad  \psi_R(z): = \psi(z) \cdot \1_{|\im(z)|\leq R}(z).$$ By Lemma \ref{kboo}
\begin{align*}
\int_E \int_{\C \setminus E} |K_l(z,w)|^{\,2} \, dA(z)\, dA(w)
& \lesssim \int_E \int_{\C \setminus E} \psi_R(z-w) \, dA(z) \, dA(w)\\
&+
\int_E \int_{\C \setminus E} (\psi-\psi_R) (z-w) \, dA(z)\,  dA(w).
\end{align*}
The first term is bounded, as before, by
\begin{align}\label{eq_a1}
\perim E \cdot \int_{\C}
|z| \psi_R(z)\, dA(z) \lesssim \log(1+R) \cdot \perim E.
\end{align}
Another elementary estimate shows that
\begin{align}\label{eq_a2}
\int_E \int_{\C \setminus E} (\psi-\psi_R) (z-w)\, dA(z)\, dA(w) \leq |E| \cdot \int_{|\im z|>R} \psi(z)\, dA(z)
\lesssim \frac  {|E|}{1+R}.
\end{align}
If $|E| \geq 2 \cdot \perim E$, choosing $R+1 = \frac{|E|}{\perim E} \geq 2$ and adding \eqref{eq_a1} and \eqref{eq_a2} yields \eqref{eq_bv_2}. On the other hand, if $|E| \leq 2 \cdot
\perim E$, then \eqref{eq_bv_2} is trivially true, since
\begin{align*}
\int_E \int_{\C \setminus E} |K_l(z,w)|^{\,2}\, dA(z)\, dA(w)
\lesssim
(\smallint \psi) \cdot |E| \lesssim \perim E.
\end{align*}
This completes the proof.
\end{proof}
We remark that the estimate \eqref{eq_bv_1} is sometimes called an ``area law''. (Compare e.g.~\cite[Theorem 1.2]{CE}.)

\section{Sampling and Interpolation} \label{sec_samp}
We now state and prove the main result on random sampling and interpolation with Coulomb systems.
Throughout this section, we assume that our external potential satisfies assumptions \ref{Q1}-\ref{Q7} and, in addition, that we are in the low temperature regime
$$\beta_n\ge c\log n$$
for some fixed $c>0$. As usual, $\family=(\family_n)_n$ denotes a random sample from the corresponding Boltzmann-Gibbs distribution, and we write $\family_n=\{\zeta_j\}_1^n$.

\begin{thm}\label{th_samp}
Fix a failure probability $\delta \in (0,1)$ and a bandwidth margin $\gamma>0$. Then there are positive constants $\newA=\newA(c)$, $M=M(c)$ and $s=s(c)$ (independent of $\delta$ and $\gamma$)
and $n_0=n_0(c,\delta, \gamma)$
such that, with probability at least $1-\delta$, the following properties hold simultaneously for all $n \geq n_0$,
	
	\noindent $\bullet$ \emph{(Width)}:
	\begin{align}\label{eq_s_pointsindroplet}
	\{\zeta_k\}_1^n \subset S_M = S + D(0,M/\sqrt{n}),
	\end{align}
	
	\noindent $\bullet$ \emph{(Separation)}:  $\conf$ is $2 s$-separated, i.e.
	\begin{align}\label{eq_separated}
	\sepn_n(\conf)= \sqrt{n}\cdot \min\big\{\,|\zetaj-\zetak|\,;\, j\ne k\,\big\} \geq 2s.
	\end{align}

	\noindent $\bullet$ \emph{(Interpolation)}: For each
	$\rho \geq 1+\gamma$ and each sequence of values $(a_j)_{j=1}^n\in\C^n$ there exists an element $f\in\calW_{n\rho}$ such that
	\begin{equation}\label{inter0}
	f(\zeta_{j})=\seq_{j},\qquad j=1,\ldots,n
	\end{equation}
	and
	\begin{equation}\label{inter}
	\int_\C|f|^{\,2}\le \frac{\newA}{n (\rho-1)^{\,2}} \sum_{j=1}^n|\seq_j|^{\,2}.
	\end{equation}
	
	\noindent $\bullet$ \emph{(Sampling)}: For each $0<\rho \leq 1-\gamma$, the following \emph{Marcinkiewicz-Zygmund} inequality holds
	\begin{equation}\label{mz}
	\int_{S_{2M}}|f|^{\,2}\le \frac{\newA}{n (1-\rho)^{\,2}} \sum_{j=1}^n|f(\zeta_{j})|^{\,2},\qquad f\in\calW_{n\rho}.
	\end{equation}
\end{thm}

\begin{rem} To avoid some uninteresting technicalities, we assume throughout that $\rho$ in (Interpolation) and (Sampling) is such that $n\rho$ is an integer. This is easiest to achieve
by allowing $\rho=\rho_n$ to depend slightly on $n$.
\end{rem}

\begin{rem}
The usual intuition (going back to Landau) is that the interpolation property implies that a family is ``sparse'', while the sampling property implies that it is ``dense''. The localization near the droplet
accounts for the fact that the $L^2$-norm in \eqref{mz} is only taken over the vicinity $S_{2M}$ of the droplet. To wit, the density of the Coulomb gas is very small
outside $S_{M}$ if $M$ is large, which is reflected by the fact that the value of our constant $\newA$ in \eqref{mz} satisfies $\newA\lr\infty$ as $M\lr\infty$.
This technical obstacle does not occur for the interpolation inequality \eqref{inter}, since the sparseness outside $S_{M}$ is (almost surely) immediate for large $M$, in view of the localization property \eqref{local0}.
\end{rem}

\begin{proof}[Proof of Theorem \ref{th_samp}]
\mbox{} \\
\noindent \emph{Step 1. (Preparations).}
Fix some bounded neighborhood $U$ of the droplet and consider the random variables
$$X_j=\int_U |\ell_j|^{\,2\beta}\, dA$$
where $\ell_j\in\calW_n$ is the weighted Lagrange polynomial associated with $\{\zeta_k\}_1^n$ as in \eqref{nota}.

Taking $f(\zeta,\zeta_j)=\1_U(\zeta_j)$ in Lemma \ref{exid} we have
$\bfE_n^{\,\beta}(X_j)=|U|$, whence
$$\bfE_n^{\,\beta}\left(\,\sum_{j=1}^n X_j\,\right)=C_1n,\qquad (C_1=|U|).$$
Using Chebyshev's inequality we obtain
\begin{align}\label{eq_aaa}
\bfP_n^{\,\beta}\left(\left\{\,\sum_1^n X_j>\lambda\,\right\}\right)\le C_1\frac n \lambda.
\end{align}

Now recall, by Lemma \ref{weh}, that for some constant $C_2$ we have the inequality
$$\|\,f\,\|_{L^\infty(S)}^{\,2\beta}\le C_2^{\,2\beta} n\cdot \|\,f\,\|_{L^{2\beta}(U)}^{\,2\beta},\qquad (f\in\calW_n).$$
Hence, by Lemma \ref{lem_maxp}, $$\|\,f\,\|_\infty^{\,2\beta}\le C_2^{\,2\beta} n\cdot \|\,f\,\|_{L^{2\beta}(U)}^{\,2\beta},\qquad (f\in\calW_n).$$
Applying this to $f=\ell_j$ and summing in $j$ we obtain
$$
\sum_{j=1}^n
\|\,\ell_j\,\|_\infty^{\,2\beta}\le C_2^{\,2\beta} n
\sum_{j=1}^n X_j.$$
Hence, by \eqref{eq_aaa}
$$\bfP_n^{\,\beta}\left(\left\{\,\sum_{j=1}^n\|\,\ell_j\,\|_\infty^{\,2\beta}>C_2^{\,2\beta} n\lambda\,\right\}\right)\le C_1\frac n\lambda.$$
Choosing $\lambda=n^{\,3}$ we find that
$$\bfP_n^{\,\beta}\left(\left\{\,\max_{1\le j\le n}\|\,\ell_j\,\|_\infty>C_2 n^{\,\frac 2\beta}\,\right\}\right)
\le \frac{C_1}{n^{\,2}}.$$
Since $\beta\ge c\log n$,
\begin{align*}
n^{\,\frac{2}{\beta}} \le e^{\,\frac{2}{c}},
\end{align*}
and we conclude that
there is a constant
$A=A(c)$ such that for all $n$
\begin{equation}\label{fina}\bfP_n^{\,\beta}\left(\left\{\,\max_{1\le j\le n}\|\,\ell_j\,\|_\infty> \newM\,\right\}\right)< \frac {C_1} {n^{\,2}}.\end{equation}

Let us fix a small failure probability $\delta>0$. By \eqref{fina}
and the Borel-Cantelli lemma, there exists $n_0=n_0(\delta)$ such that
with probability at least $1-\delta/2$ we have, \emph{for all} $n \geq n_0$,
\begin{equation}\label{fin}
\max_{1\le j\le n}\|\,\ell_j\,\|_\infty \leq \newM.
\end{equation}

By Theorem \ref{mth} and \eqref{local0}, $n_0$ can be chosen so that, in addition, with probability at least $1-\delta/2$ the properties (width) and (separation) are satisfied for adequate constants, so that all three properties (width), (separation) and \eqref{fin}
hold
with probability at least $1-\delta$ when $n\ge n_0$.
By suitably enlarging $n_0$ (depending on $\gamma$), we further assume that
\begin{align}\label{eq_n0}
n_0 \geq \tfrac{4}{\gamma}.
\end{align}

Fix $n \geq n_0$ and a configuration $\conf$
for which (width), (separation) and \eqref{fin}
hold. Let us verify the remaining properties (interpolation) and (sampling).

\smallskip

\noindent \emph{Step 2. (The Coulomb gas as an interpolating family)}.
Let us choose a number $\rho \geq 1+\gamma$ and write $\rho=1+2\eps$. We may assume that $n\eps$ is an integer.
By \eqref{eq_n0}, $n_0 \geq \tfrac{2}{\eps}$. Take $n\ge n_0$ and let $\bfK_{n\eps}$ be the reproducing kernel
of the space $\calW_{n\eps}$ (equipped with the norm of $L^2$).

We form new weighted polynomials $L_j\in \calW_{n\rho}$ by multiplying with a localizing factor as follows,
$$L_j(\zeta)=\left(\frac {\bfK_{n\eps}(\zeta,\zetaj)}{\bfK_{n\eps}(\zetaj,\zetaj)}\right)^{\, 2}\cdot \ell_j(\zeta).$$

In view of Theorem \ref{lubb}, \ref{harder} and the assumption $\{\zeta_j\}_1^n\subset S_M$ there is a constant $c_1>0$ (independent of $\eps$) such that
\begin{equation}\label{ln3}\bfK_{n\eps}(\zeta_j,\zeta_j)\ge c_1n\eps,\qquad j=1,\ldots,n.\end{equation}
Likewise by Theorem \ref{lubb}, \ref{trivial} there is a uniform upper bound
\begin{equation}\label{ln4}\bfK_{n\eps}(\zeta,\zeta)\le c_2n\eps,\qquad  (\zeta\in\C).\end{equation}

Now recall the Berezin kernel $\bfB_{n\eps}(\zeta_j,\zeta)$,
$$\bfB_{n\eps}(\zeta_j,\zeta)=\frac {|\bfK_{n\eps}(\zeta_j,\zeta)|^{\,2}}{\bfK_{n\eps}(\zeta_j,\zeta_j)}.$$
This is a probability density in $\zeta$, i.e., $\int_\C \bfB_{n\eps}(\zeta_j,\zeta)\, dA(\zeta)=1$, and we have by \eqref{fin}, \eqref{ln3}
\begin{equation}\label{basset}|L_j(\zeta)|=\frac {\bfB_{n\eps}(\zeta_j,\zeta)}{\bfK_{n\eps}(\zeta_j,\zeta_j)}|\ell_j(\zeta)|\le \frac A {c_1n\eps} \bfB_{n\eps}(\zeta_j,\zeta),\end{equation}
and
hence
\begin{equation}\label{vomb}\|\,L_j\,\|_1\le\frac C {n\eps}\end{equation}
where $C$ is independent of $n$, $j$, and $\eps$.

Next write $\bfK_{n\eps,\zeta}\in\calW_{n\eps}$ for the reproducing kernel
$$\bfK_{n\eps,\zeta}(\eta)=\bfK_{n\eps}(\eta,\zeta).$$

Using in turn: \eqref{basset} and the lower bound \eqref{ln3}, the uniform separation \eqref{eq_separated} and Corollary \ref{v2}, the reproducing property, and the upper bound \eqref{ln4},
we find for all $\zeta\in\C$
\begin{align}\label{nu}
\begin{aligned}
\sum_{j=1}^n|L_j(\zeta)|&\le \frac {C_1}{\eps^{\,2} n^{\,2}}\sum_{j=1}^n|\bfK_{n\eps,\zeta}(\zetaj)|^{\,2}
\\
&\le \frac{C_2}{\eps^{\,2} n}\int_{\C}|\bfK_{n\eps,\zeta}(\eta)|^{\,2}\, dA(\eta)\\
&=\frac {C_2}{\eps^{\,2} n}\bfK_{n\eps}(\zeta,\zeta)\le \frac{C}{\eps}.
\end{aligned}
\end{align}

Now define a linear operator $T:\C^n\to (L^1+L^\infty)(\C)$ by
$$T(\seq)=\sum_{j=1}^n \seq_j L_j,\qquad \seq=(a_j)_1^n \in \C^n.$$
Then by \eqref{vomb} and \eqref{nu},
\begin{equation}\label{end1}\|\,T\,\|_{\ell^1_n\to L^1}\le \max_{1\le j\le n}\|\,L_j\,\|_1\le \frac {C}{\eps n},\end{equation}
while
\begin{equation}\label{end2}\|\,T\,\|_{\ell^\infty_n\to L^\infty}\le
\left\|\,
\sum_{j=1}^n|L_j|
\,\right\|_\infty\le \frac C \eps.\end{equation}
An application of the Riesz-Thorin theorem gives
$$\|\,T\,\|_{\ell^2_n\to L^2}\le\frac C {\eps \sqrt{n}}.$$
If we set $f=T(\seq)$ this means that $f\in \calW_{n\rho}$ satisfies $f(\zetaj)=\seq_j$ for all $j$ and
\begin{equation*}
\int_\C|f|^{\,2}\, dA\le \frac {C^{\,2}} {\eps^{\,2} n}\sum_{j=1}^n|a_j|^{\,2}
\leq
\frac {C_3} {(\rho-1)^{\,2} n}\sum_{j=1}^n|a_j|^{\,2},
\end{equation*}
which proves \eqref{inter}.

\smallskip

\noindent \emph{Step 3. (The Coulomb gas as a Marcinkiewicz-Zygmund family)}.
Let us choose $0 < \rho \leq 1 - \gamma$ and write $\rho=1-2\eps$, where we again may assume that $n\eps$ is an integer.

For fixed $\zeta\in S_{2M}$ and
$f\in \calW_{n\rho}$ we define a weighted polynomial $g_\zeta\in \calW_n$ by
$$g_\zeta(\eta)=f(\eta)\cdot\left(\frac {\bfK_{n\eps}(\eta,\zeta)}{\bfK_{n\eps}(\zeta,\zeta)}\right)^{\,2}.$$

For any element of $\calW_n$ we have the Lagrange interpolation formula
$$g_\zeta(\eta)=\sum_{j=1}g_\zeta(\zetaj)\cdot \ell_j(\eta).$$
Putting $\eta=\zeta$ this gives
\begin{equation}\label{inf0}f(\zeta)=g_\zeta(\zeta)=\sum_{j=1}^n f(\zetaj)\cdot \tilde{L}_j(\zeta),\end{equation}
where
$$\tilde{L}_j(\zeta)=\left(\frac {\bfK_{n\eps}(\zetaj,\zeta)}{\bfK_{n\eps}(\zeta,\zeta)}\right)^{\,2}\cdot\ell_j(\zeta)
.$$
The lower bound in Theorem \ref{lubb} gives that there is a constant $c_1=c_1(M)>0$ such that
$$\bfK_{n\eps}(\zeta,\zeta)\ge c_1n\eps\qquad \text{when}\qquad \zeta\in S_{2M}.$$
Combining this with the estimate in \eqref{fin}, we find that
$$|\tilde{L}_j(\zeta)|\le\frac A {(c_1n\eps)^{\,2}}|\bfK_{n\eps}(\zeta,\zeta_j)|^{\,2},\qquad (\zeta\in S_{2M}).$$
The reproducing property $\int \bfK_{n\eps}(\zeta',\zeta) \, \bfK_{n\eps}(\zeta,\zeta'') \, dA(\zeta) = \bfK_{n\eps}(\zeta',\zeta'')$ thus gives
\begin{align*}\int_{S_{2M}}|\tilde{L}_j(\zeta)|\, dA(\zeta)&\le \frac A {(c_1n\eps)^{\,2}}\int_{S_{2M}}|\bfK_{n\eps}(\zeta,\zeta_j)|^{\,2}\, dA(\zeta)\\
&\le \frac A{(c_1n\eps)^{\,2}}\bfK_{n\eps}(\zeta_j,\zeta_j)\\
&\le \frac {C_1} {n\eps},\end{align*}
where we again used the uniform upper bound \eqref{ln4} to obtain the last inequality.

We now put
$$\tilde{F}(\zeta)=\sum_{j=1}^n|\tilde{L}_j(\zeta)|$$
and observe that (by Corollary \ref{v2} which is applicable due to the uniform separation of $\conf$),
\begin{align*}\tilde{F}(\zeta)&\le\frac {A} {(c_1n\eps)^{\,2}}\sum_{j=1}^n|\bfK_{n\eps,\zeta}(\zetaj)|^{\,2}\\
&\le \frac {C_1 \newM}{c_1^{\,2}\eps^{\,2} n} \int_\C|\bfK_{n\eps,\zeta}|^{\,2}\\
&=\frac {C_1 \newM}{c_1^{\,2}\eps^{\,2} n}\bfK_{n\eps}(\zeta,\zeta)\le \frac {C} \eps,
\end{align*}
by virtue of the upper bound \eqref{ln4}.

Consider the linear operator $\tilde{T}:\C^n\to (L^1+L^\infty)(S_{2M})$ given by
$$\tilde{T}(\seq)=\sum_{j=1}^n \seq_j\tilde{L}_j.$$
The above estimates show that $$\|\,\tilde{T}\,\|_{\ell^1_n\to L^1(S_{2M})}\le \frac{C}{\eps n},\qquad \|\,\tilde{T}\,\|_{\ell^\infty_n\to L^\infty(S_{2M})}\le \frac{C}{\eps},$$
so by the Riesz-Thorin theorem,
$$\|\,\tilde{T}\,\|_{\ell^2_n\to L^2(S_{2M})}\le\frac C {\eps \sqrt{n}}.$$

By \eqref{inf0}, any $f\in\calW_{n\rho}$ can be represented as $f=\tilde{T}(c)$ where $\seq_j=f(\zetaj)$ for $j=1,\ldots,n$. Hence,
\begin{equation*}\int_{S_{2M}}|f|^{\,2}\le \frac {C^{\,2}}{\eps^{\,2} n}\sum_{j=1}^n |f(\zetaj)|^{\,2}\leq \frac{C_3}{(1-\rho)^{\,2} n}\sum_{j=1}^n |f(\zetaj)|^{\,2}
.
\end{equation*}
This, proves \eqref{mz}, thus finishing our proof of Theorem \ref{th_samp}.
\end{proof}

\begin{rem} While our main focus here is on the analysis of random
configurations, the above proof of Theorem \ref{th_samp} also applies to deterministic
configurations and shows that the conclusions hold under suitable
separation and density properties as in Theorem \ref{mth} and Theorem \ref{mth2}, as
these lead to the bounds for Lagrange polynomials in \eqref{fin}. An infinite
dimensional counterpart of such result is found in \cite{BOC}; see also \cite{Se}.
\end{rem}

\section{Equidistribution} \label{sec_equ}

In this section we prove Theorem \ref{mth2} and Proposition \ref{propop}.
As in \cite{AOC,A0} we largely follow Landau's strategy from his work \cite{L}
on interpolation and sampling in Paley-Wiener spaces, but with certain basic modifications due to the localization to the vicinity of the droplet.

Throughout this section, we fix a potential $Q$ which satisfies all the assumptions \ref{Q1}-\ref{Q7}.
We will write $$\langle\, f\,,\,g\,\rangle=\int_\C f\bar{g}\, dA$$ for the inner product in the space $L^2=L^2(\C,dA)$.
We shall regard the space $\calW_n$ of weighted polynomials as a subspace of $L^2$.

\subsection{Concentration operators} Given a domain $\Omega\subset\C$, the
corresponding ``concentration operator'' $T_{n,\Omega}$
is the Toeplitz operator on $\calW_n$ defined by
\begin{align}\label{eq_conop}
T_{n,\Omega} f =
P_{\calW_{n}} (f\cdot \1_{\Omega}), \qquad (f\in\calW_n),
\end{align}
where $P_{\calW_{n}}$ is the orthogonal projection of
$L^2$ onto $\calW_{n}$. Thus $T_{n,\Omega}$ is a (strictly) positive contraction, and we can write its eigenvalues in non-increasing order as
$$1\ge \lambda_1\ge \lambda_2\ge\cdots\ge\lambda_n>0.$$

\begin{lem}\label{lemma_eig} Fix a number $\vt$, $0<\vt<1$. Then
\begin{align*}
\left|\,
\# \left\{\, j\, ; \lambda_j \ge \vt\, \right\} - \trace T_{n,\Omega}\, \right| \leq \max\left\{\tfrac{1}{\vt},\tfrac{1}{1-\vt}\right\} \cdot \big[\trace T_{n,\Omega}-\trace T^{\,2}_{n,\Omega} \big].
\end{align*}
\end{lem}
\begin{proof}
	Observe that $$\# \left\{\,j\, ;\, \lambda_j \ge \vt\, \right\} - \trace T_{n,\Omega} =
	\trace \psi(T_{n,\Omega}) $$ where
	\begin{align*}
	\psi(t) := \left\{
	\begin{array}{ll}
	-t, & \mbox{if } 0 \leq t < \vt,\\
	1-t, & \mbox{if } \vt \le  t \leq 1.
	\end{array}
	\right.
	\end{align*}
	and use the estimate $|\psi(t)| \leq \max\{\tfrac{1}{\vt},\tfrac{1}{1-\vt}\}
	(t-t^{\,2})$ for $t \in [0,1]$.
\end{proof}

In the following, we shall consider blow-ups about (perhaps $n$-dependent) points $p_n$. The following lemma will come in handy. (See \cite{AKMW} for related statements, valid near cusps.)

\begin{lem} \label{regular} Let $p_n\in \d S$ be a boundary point and let $e^{\, i\theta_n}$ be the direction of the normal of $\d S$ at the point $p_n$, pointing outwards from $S$.
Fix a parameter $\rho$ with $0<\rho<2$ and let $\magn_{n\rho}$ be the corresponding magnification map:
\begin{equation}\magn_{n\rho}:\label{magn}\zeta\longmapsto z,\qquad z=\magn_{n\rho}(\zeta)=\sqrt{n\rho\Delta Q(p_n)}\cdot e^{\,-i\theta_n}\cdot (\zeta-p_n).\end{equation}
Then for each (large) $C>0$, the indicator function $\1_{\magn_{n\rho}(S)\cap D(0,C)}$ converges to $\1_{\L\cap D(0,C)}$ in the norm of $L^1$, where $\L$ is the left half-plane,
$$\L=\{\, z\,;\,\re z\le 0\}.$$
\end{lem}

\begin{proof} Recall that our assumptions on $Q$ imply (via Sakai's theory) that the boundary $\d S$ is everywhere real-analytic.

We may assume that $p_n=0$ and $\theta_n=0$, i.e., the boundary $\d S$ is tangential to the imaginary axis at $0$.
There is then an $\epsilon>0$ such that
the portion of $\d S$ inside $D(0,\epsilon)$ is given by a graph
$$u=c_2v^2+c_3v^3+\cdots,\qquad \zeta=u+iv\in (\d S)\cap D(0,\epsilon).$$
Writing $z$ in \eqref{magn} as $z=x+iy$, we see that the image of the curve $(\d S)\cap D(0,\epsilon)$ is
$$x=c_2'n^{-1/2}y^2+c_3'n^{-1}y^3+\cdots$$
for suitable coefficients $c_2',c_3',\cdots$. Now fix a large $C>0$ and consider the set
$$S_{n,C}=D(0,C)\cap \magn_{n\rho}(S)=\left\{\, z\in D(0,C)\, ;\, x\le c_2'n^{-1/2}y^2+c_3'n^{-1}y^3+\cdots\,\right\}.$$
It is clear that $\1_{S_{n,C}}$ converges to $\1_{D(0,C)\cap\{x\le 0\}}$ in the norm of $L^1$.
\end{proof}

We will also need an asymptotic description of the quantities in Lemma \ref{lemma_eig}. The following Lemma is
essentially found in \cite[Lemmas 4.1 and 4.2]{A0}, but we shall supply some extra details about the proof.

\begin{lem}\label{lemma_traces} Fix a sequence $\bfp=(p_n)_1^\infty$ which belongs to $S_M=S+D(0,M/\sqrt{n})$ for some $M>0$.
Assume that the limit $p_*=\lim p_n$ exists (along some subsequence).
Also fix numbers $L\ge 2$ and $\rho$ with $0<\rho<2$ and consider the concentration operator $T=T(\rho, n,p_n,L,M)$ defined by
$$T=T_{\rho n,\Omega}:\calW_{\rho n}\to\calW_{\rho n},\qquad Tf=P_{\calW_{n\rho}}(f\cdot\1_\Omega),\qquad \Omega=D(p_n,L/\sqrt{n})\cap S_M.$$
Then (along a further subsequence)
\begin{align*}
\lim_{n \lr \infty}
\trace T
= \begin{cases}
\rho \cdot \Delta Q (p_*) \cdot L^{\,2} &\mbox{bulk case} \\
\tfrac{\rho}{2} \cdot \Delta Q (p_*) \cdot L^{\,2} + O(L) &\mbox{boundary case}
\end{cases}
\end{align*}
and
\begin{align*}
\lim_{n \lr \infty}
\trace \left(
T - T^{\,2}
\right)
=
\begin{cases}
O(L) &\mbox{bulk case} \\
O(L \log L) &\mbox{boundary case}
\end{cases}
\end{align*}
where the implied constants depend only on $Q$ and $M$.
\end{lem}

\begin{proof}
By passing to a suitable subsequence, we can assume that $(p_n)$ is either in the bulk regime or in the boundary regime, and that the limit
\begin{equation}\label{newlimit}l=\lim_{n\to\infty}\sqrt{n}\cdot e^{\,-i\theta_n}\cdot (p_n-q_n)\end{equation}
exists, where $q_n\in \d S$ is the closest point to $p_n$ and $e^{\,i\theta_n}$ is the outwards unit normal to $\d S$ at $q_n$. (Note that $l\le M$ and that the bulk case
corresponds to $l=-\infty$.)

It is easy to see that
\begin{align*}
\trace T
&= \int_{\Omega} \bfK_{\rho n}(\zeta,\zeta)\, dA(\zeta),
\\
\trace T^{\,2}
&= \iint_{\Omega^{\,2}} |\bfK_{\rho n}(\zeta,\eta)|^{\,2}\, dA_2(\zeta,\eta).
\end{align*}
We now zoom on the point $p_n$ using the magnification map $z=\magn_{n\rho}(\zeta)$ from \eqref{magn}, with the following convention about angles $\theta_n$: we put $\theta_n=0$ if
$\bfp$ is in the bulk regime and $e^{\,i\theta_n}$ is the outwards unit normal to $\d S$ at $q_n$. (This is in accordance with the earlier convention in Section \ref{FPD}.)

Similar as in Section \ref{FPD}, we put
$$K_{\rho n}(z,w)=\frac 1 {n\rho\Delta Q(p_n)}\bfK_{n\rho}(\zeta,\eta),\qquad z=\magn_{n\rho}(\zeta),\quad w=\magn_{n\rho}(\eta),$$
and observe that
\begin{align}\label{eq_trace1}
\trace T& =
\int_{\magn_{n\rho}(\Omega)} K_{n\rho}(z,z)\, dA(z),\\
\label{eq_trace2}
\trace T^{\,2}
&= \iint_{\magn_{n\rho}(\Omega)^{\,2}}
| K_{n\rho}(z,w)|^{\,2}\, dA_2(z,w).
\end{align}

Now write $d=\sqrt{\rho\Delta Q(p_*)}$ and let $\L_1$ be the translated half-plane
$$\L_1=\L-l\cdot d+M\cdot d=\left\{\,z\, ;\, \re (z+l\cdot d)\le M\cdot d\,\right\},$$
and set
\begin{align*}
E_1(L)&=D(0,L\cdot d),\\
E_2(L)&=D(0,L\cdot d)\cap(\L_1).
\end{align*}
By Lemma \ref{regular} and an elementary geometric consideration, we see that the characteristic function $\1_{\magn_{\rho n}(\Omega)}$
converges in the $L^1$-sense to $\1_{E_1(L)}$
in the bulk case, and
to $\1_{E_2(L)}$ in the boundary case.

We now use Theorem \ref{prop_univ} (with $n\rho$ in place of $n$), to take the limit on \eqref{eq_trace1}, and obtain
\begin{align*}
\lim_{n\lr\infty} \trace  T = \int_{E(L)} R\, dA
\end{align*}
where $R\equiv 1$ and $E(L)=E_1(L)$ in the bulk case while $R(z)=F(z+\bar{z}+2ld)$ and $E(L)=E_2(L)$ in the boundary case, respectively.
(As always, $F$ denotes the holomorphic $\erfc$-kernel from \eqref{eq_fbf}.)

Thus, in the bulk case we have
$$\lim_{n\lr\infty} \trace T =|E_1(L)|= \rho \cdot \Delta Q(p_*)\cdot  L^{\,2}.$$

Similarly, an easy computation using
asymptotics for the $\erfc$-kernel shows that, in the boundary case,
\begin{align*}
\lim_{n\lr\infty} \trace T = \int_{E_2(L)}  F(z+\bar{z}+2ld) \,dA(z) = \frac{1}{2}  \rho\cdot \Delta Q(p_*)\cdot L^{\,2} + O(L),\qquad (L\to\infty),
\end{align*}
where the implied constant depends on $M$ (and the potential $Q$).

Now let $K=GL$ be a limiting kernel in Lemma \ref{lemlim}, so $K(z,z)=R(z)$ with $R$ as above. (So $K=G$ in the bulk case
and $K=K_l$ in the boundary case.) Then (along the relevant subsequence)
\begin{align*}
\lim_{n\lr\infty} \trace \left(T-T^{\,2}\right)
&= \int_{E(L)} R - \iint_{E(L)^{\,2}} |K(z,w)|^{\,2}\, dA_2(z,w)
\\
&
= 2\int_{E(L)} \int_{\C \setminus E(L)} |K(z,w)|^{\,2}\, dA_2(z,w).
\end{align*}
The desired bounds now follow from Lemma \ref{lemma_perims} on noting that
$$\perim E_1(L) \asymp \perim E_2(L) \lesssim L,\qquad |E_2(L)| \lesssim L\cdot \perim E_2(L)$$
where (for $L\ge 2$) the implied constants depend only on $M$ and $Q$.
\end{proof}

\subsection{Equidistribution and discrepancy} We now prove Theorem \ref{mth2} on equidistribution and Proposition \ref{propop} about discrepancy estimates. While the literature
on density conditions for sampling and interpolation is ample, Landau's original method
seems to adapt best to partial Marcinkiewicz-Zygmund inequalities \eqref{mz}.
In dealing with certain technicalities we also benefited from reading \cite{ACNS,NO,RS95}.

To get started, we fix a sequence $\bfp=(p_n)$ such that each $p_n$ is
contained in $S_M=S+D(0,M/\sqrt{n})$ for some $M>0$.
After passing to a subsequence we can assume that $p_n$ converges to some point $p_*\in S$.
(Recall that the exterior case was already disposed of after the statement of Theorem \ref{mth2}.)

We fix $L>2$, a failure probability $\delta \in (0,1)$, and a bandwidth margin $\gamma \in (0,1)$, and invoke Theorem \ref{th_samp}.
Let $M=M(c)$, $s=s(c)$, $\newA=\newA(c)$, and $n_0=n_0(c,\delta,\gamma)$ be the respective constants. We then select with probability
at least $1-\delta$ a family $\family=(\family_n)_n$ such that the samples
$$\family_n=\{\zeta_j\}_1^n$$
satisfy all conditions in Theorem \ref{th_samp} when $n\ge n_0$.
Below, we fix $n\ge n_0$ and let $\{\zeta_j\}_1^n$ be a configuration satisfying those conditions. (We may also allow $\gamma$ to be slightly
$n$-dependent, so we can assume that $n\gamma$ is an integer).

We may assume without loss of generality that $s<M$ and $s<1/4$, and also $n_0 \gamma \geq 2$. In what follows, all implied constants are allowed to depend on $c$ and $Q$. An unspecified norm $\|\cdot\|$
will always denote the norm in $L^2(\C,dA)$.

To simplify the notation we write
\begin{align*}
&D=D\left(\,p_n, L/\sqrt{n}\,\right), \\
&D^{\,+}=D\left(\,p_n, (L+s)/\sqrt{n}\,\right), \quad D^{\,-}=D\left(\,p_n, (L-s)/\sqrt{n}\,\right),\\
&N_n= \#\left(\left\{\,\{\zeta_j\}_1^n\cap D\,\right\}\right) = \# \left(\left\{\,\{\zeta_j\}_1^n\cap D \cap S_M\right\}\right), \\
&N_n^{\,\pm}=
\# \left(\left\{\,\{\zetaj\}_1^n\cap D^{\,\pm}\,\right\}\right) = \#\left( \left\{\,\{\zeta_j\}_1^n\cap D^{\,\pm} \cap S_M\,\right\}\right),
\end{align*}
where we used \eqref{eq_s_pointsindroplet}. Due to the $2s$-separation, we have
\begin{align}\label{eq_plus}
N_n^{\,-} \leq N_n \leq N_n^{\,+} \leq N_n^{\,-} + CL,
\end{align}
for a constant $C=C(M,s)$.

\noindent {\em Step 1. (Lower density bounds)}. Choose
$$\rho=1-\gamma$$
and consider the concentration operator
\begin{equation}\label{conop}T:\calW_{n\rho}\to\calW_{n\rho},\qquad f\longmapsto P_{\calW_{n\rho}}(f\cdot\1_{D\cap S_{M+s}}).\end{equation}

Let $(\phi_j)_1^{n\rho}$ be an orthonormal basis for $\calW_{n\rho}$ consisting of eigenfunctions of $T$,
$T(\phi_j)=\lambda_j\phi_j,$
where, as before, we use the convention $1\ge \lambda_1\ge \lambda_2\ge\cdots\ge \lambda_n>0$.

Write
$$F_n=\lspan\{\,\phi_1,\ldots,\phi_{N_n^++1}\,\}.$$
We can then find an element $f\in F_n$ with $\|\,f\,\|=1$ which vanishes at each point in $\{\zeta_j\}_1^n\cap D^{\,+}$.

Since $\{\zeta_j\}_1^n$ is $2s$-separated, the MZ inequality \eqref{mz}
and Corollary \ref{v2} imply
\begin{align*}
\int_{S_{M+s}} |f|^{\,2} \leq \frac{\newA}{n(1-\rho)^{\,2}} \sum_{\zetaj \notin D^{\,+}} |f(\zetaj)|^{\,2} \leq \frac{C}{\gamma^{\,2}} \int_{S_{M+s} \setminus D} |f|^{\,2}.
\end{align*}
Hence,
\begin{align*}
\int_{S_{M+s} \cap D} |f|^{\,2}=
\int_{S_{M+s}} |f|^{\,2}- \int_{S_{M+s} \setminus D} |f|^{\,2}
\leq \big(1-\tfrac{\gamma^2}{C} \big) \int_{S_{M+s}} |f|^{\,2} \leq 1-\tfrac{\gamma^2}{C}.
\end{align*}
On the other hand,
\begin{align*}
\lambda_{N_n^++1} \leq \langle\, T f\,,\, f\, \rangle = \int_{S_{M+s} \cap D} |f|^{\,2}.
\end{align*}
Therefore,
\begin{align}\label{eq_therefore_1}
\lambda_{N_n^{\,+}+1} \leq 1-\tfrac{\gamma^{\,2}}{C}.
\end{align}
We may assume that
$C>2$, so that $1-\tfrac{\gamma^{\,2}}{C}>1/2$.
An application of Lemma \ref{lemma_eig} (with $\vt=1-\gamma^2/C$) then gives
\begin{align*}
N_n^{\,+}+1 \geq \mathrm{\trace}T -
\frac{C}{\gamma^{\,2}} \cdot
\big[ \mathrm{\trace}T-
\mathrm{\trace}T^{\,2} \big].
\end{align*}
We now apply Lemma \ref{lemma_traces}, with
$M+s$ in lieu of $M$.
Combining with \eqref{eq_plus} yields
\begin{align}\label{eq_rho_1}
\liminf_{n \lr \infty} N_n \geq
\begin{cases}
(1-\gamma) \cdot \Delta Q(p_*) \cdot L^{\,2} +O\big(\gamma^{\,-2} L\big)&\mbox{bulk case} \\
\tfrac{(1-\gamma)}{2} \cdot \Delta Q(p_*) \cdot L^{\,2} + O\big(\gamma^{\,-2} L\log L\big) &\mbox{boundary case}
\end{cases}
\end{align}
where the implied constants are independent of $\gamma$.

(To be precise, in order to obtain \eqref{eq_rho_1}, we first
assume that $\gamma \in \mathbb{Q}$ and select a subsequence $(n_k)$ such that $\lim_{k\to\infty}
N_{n_k}=\liminf_{n\to\infty} N_{n}$ and $\rho n_k \in \mathbb{N}$,
and then apply Lemma \ref{lemma_traces} to this subsequence.)

\noindent {\em Step 2. (Upper density bounds)}. This time we set
$$\rho=1+\gamma.$$
We consider again the concentration operator $T$ from \eqref{conop}.

For $j=1,\ldots,n$ consider the reproducing kernels $\bfK_{\zeta_j}\in\calW_{n\rho}$,
$$\bfK_{\zeta_j}(\zeta)=\bfK_{n\rho}(\zeta,\zeta_j).$$
Consider the subspace $V_1$ of $\calW_{n\rho}$ spanned by these elements
$$V_1=\lspan\left\{\,\bfK_{\zeta_1},\ldots,\bfK_{\zeta_n}\,\right\}$$
and the orthogonal complement
\begin{align*}
V &:= V_1 \ominus
\lspan \left\{\,\bfK_{\zeta_j}\, ; \zetaj \notin D^{\,-}\,\right\}.
\end{align*}
Notice that
$\dim V = N_n^{\,-}.$

Now pick an element $f \in V$. Since the family $\family$ is assumed to have the
interpolation property in Theorem \ref{th_samp},
there exists an element $f_1 \in \calW_{\rho n}$ such that
$f_1(\zetaj)=f(\zetaj)$, for all $j=1, \ldots, n$ and
\begin{align}\label{eq_norm_fj}
\|\,f_1\,\|^{\,2} \leq \frac{\newA}{n (\rho-1)^{\,2}}
\sum_{j=1}^{n} |f(\zetaj)|^{\,2}=\frac{\newA}{n (\rho-1)^{\,2}}
\sum_{\zeta_j\in D^{\,-}} |f(\zetaj)|^{\,2}.
\end{align}
where we used that $f(\zetaj)=\langle \,f\,,\, \bfK_{\zeta_j}\,\rangle=0$, if $\zetaj \notin D^{\,-}$.
Combining with the $2s$-separation and applying Corollary \ref{v2}, we obtain
\begin{align}\label{nytt}
\|\,f_1\,\|^{\,2} \leq \frac{C}{\gamma^{\,2}} \int_{D \cap S_{M+s}} |f|^{\,2}.
\end{align}

Letting $P_{V_1}:\calW_{n\rho}
\to V_1$ be the orthogonal projection, we note that
$$P_{V_1} f_1(\zeta_j)=\langle\, P_{V_1}f_1\,,\,\bfK_{\zeta_j}\,\rangle=\langle\, f_1\,,\,P_{V_1}\bfK_{\zeta_j}\,\rangle=f_1(\zeta_j)=f(\zeta_j),\qquad j=1,\ldots,n,$$
and $\|\,P_{V_1}f_1\,\|\le\|\,f_1\,\|$. Replacing $f_1$ by $P_{V_1}f_1$, we can thus assume besides \eqref{nytt} that $f_1\in V_1$. As the mapping
$$V_1\lr \C^n,\qquad g\longmapsto (g(\zeta_j))_{j=1}^n=(\langle\, g\,,\,\bfK_{\zeta_j}\,\rangle)_{j=1}^n$$
is a linear bijection, we conclude that $f=f_1$.

In conclusion, we obtain
\begin{align}\label{eq_landau2}
\|\,f\,\|^{\,2} \leq \frac{C}{\gamma^{\,2}} \int_{D \cap S_{M+s}} |f|^{\,2}, \qquad f \in V.
\end{align}
On the other hand, since $\dim V = N_n^{\,-}$, by the Courant-Fischer characterization of eigenvalues of self-adjoint operators,
\begin{align}\label{eq_landau1}
\lambda_{N_n^{\,-}} \geq \min_{ f \in V \setminus\{0\} } \frac{ \langle\, T f\,,\,f\, \rangle}{\|\,f\,\|^{\,2}} =
\min_{ f \in V \setminus\{0\} } \frac{1}{\|\,f\,\|^{\,2}}
\int_{S_{M+s} \cap D} |f|^{\,2} \geq \frac{\gamma^{\,2}}{C}.
\end{align}
Assuming again as we may that $C>2$, it follows that $\gamma^{\,2}/C < 1/2$ and Lemma \ref{lemma_eig} (with $\vt=\gamma^2/C$)
yields
\begin{align*}
N_n^{\,-} \leq \mathrm{\trace} T +
\frac{C}{\gamma^{\,2}} \cdot
\mathrm{\trace}(T-T^{\,2}).
\end{align*}
We now apply Lemma \ref{lemma_traces}, with
$M+s$ in lieu of $M$. Combined with \eqref{eq_plus} this yields
\begin{align}\label{eq_rho_2}
\limsup_{n \lr \infty} N_n \leq
\begin{cases}
(1+\gamma) \cdot \Delta Q(p_*) \cdot L^{\,2} +O\big(\gamma^{\,-2} L\big)&\mbox{bulk case} \\
\tfrac{(1+\gamma)}{2} \cdot \Delta Q(p_*) \cdot L^{\,2} + O\big(\gamma^{\,-2} L\log L\big) &\mbox{boundary case}
\end{cases}
\end{align}
(Again, the precise derivation of \eqref{eq_rho_2} is a follows: we first select a subsequence $(n_k)$ such that $\lim_{k\to\infty}
N_{n_k}=\limsup_{n\to\infty} N_{n}$, and then apply Lemma \ref{lemma_traces} to this subsequence.)

\noindent {\em Step 3. (Conclusions)}.
Combining \eqref{eq_rho_1} and \eqref{eq_rho_2} we obtain
\begin{align}\label{limsups}
\begin{cases}
\limsup_{n\lr\infty} \left|\,N_n -\Delta Q(p_*) \cdot L^{\,2}\,\right|
= O\big(\gamma  L^{\,2} + \gamma^{-2} L\big)&\mbox{bulk case} \\
\limsup_{n\lr\infty} \left|\, N_n-\tfrac{1}{2} \cdot \Delta Q(p_*) \cdot L^{\, 2}\, \right| = O\big(\gamma L^{\,2} + \gamma^{-2} L\log L\big) &\mbox{boundary case}
\end{cases}
\end{align}
where the implied constants are independent of the failure probability $\delta$. Letting $\delta \lr 0$, we infer that
\eqref{limsups} hold for almost every family, picked randomly with respect to the Boltzmann-Gibbs measure.

Finally, taking $\gamma=L^{-1/3}$ yields the desired discrepancy estimates in Proposition \ref{propop}, from which the claims on Beurling-Landau densities in Theorem \ref{mth2} also follow.
By this, all statements are proved. Q.E.D.

\section{Concluding remarks} \label{CORE}

\begin{figure}[h]
	\begin{subfigure}[h]
{0.32\textwidth}
		\begin{center}
			\includegraphics[width=1.69in,height=1.69in]{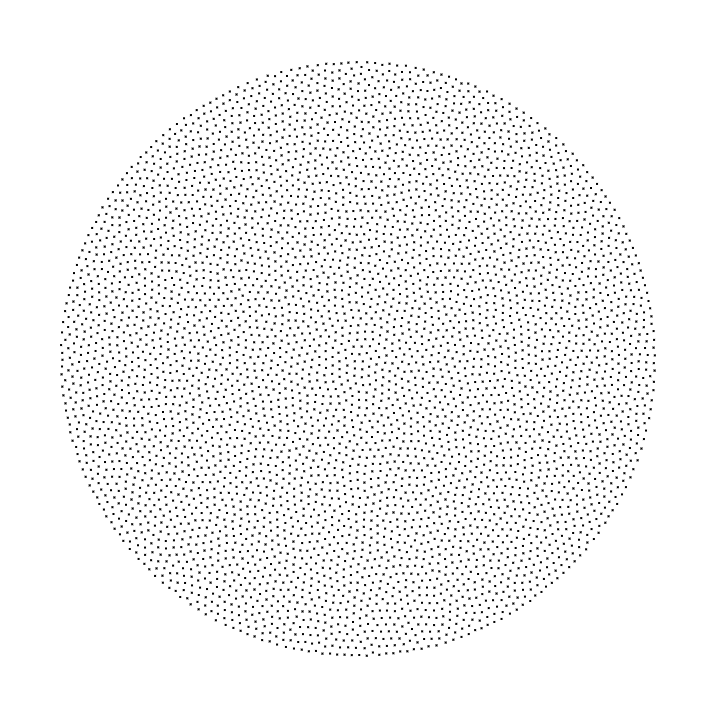}	
		\end{center}
		\caption{$Q=|\zeta|^2$}
	\end{subfigure}
	\begin{subfigure}[h]
{0.32\textwidth}
		\begin{center}
			\includegraphics[width=1.69in,height=1.69in]{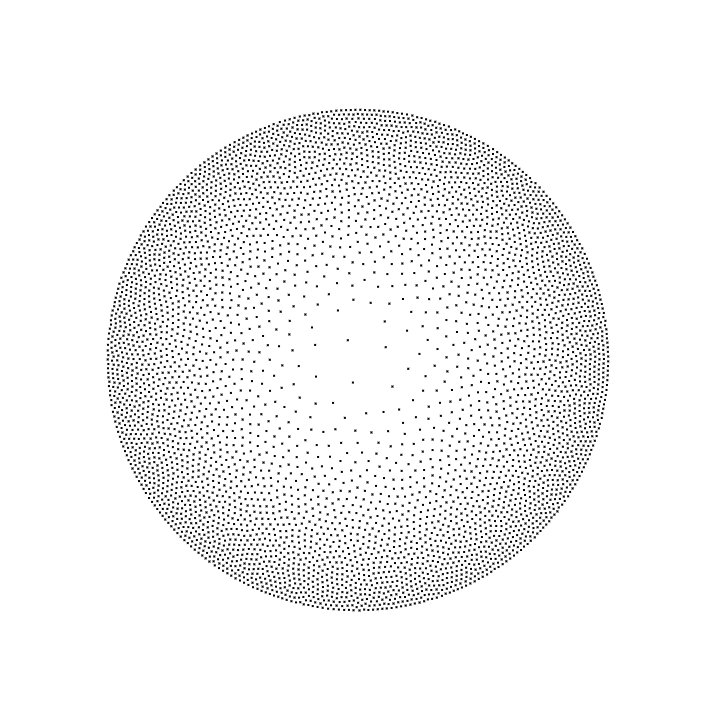}
		\end{center}
		\caption{$Q=|\zeta|^4$}
	\end{subfigure}	
\\
	\begin{subfigure}[h]
{0.32\textwidth}
		\begin{center}
			\includegraphics[width=1.69in,height=1.69in]{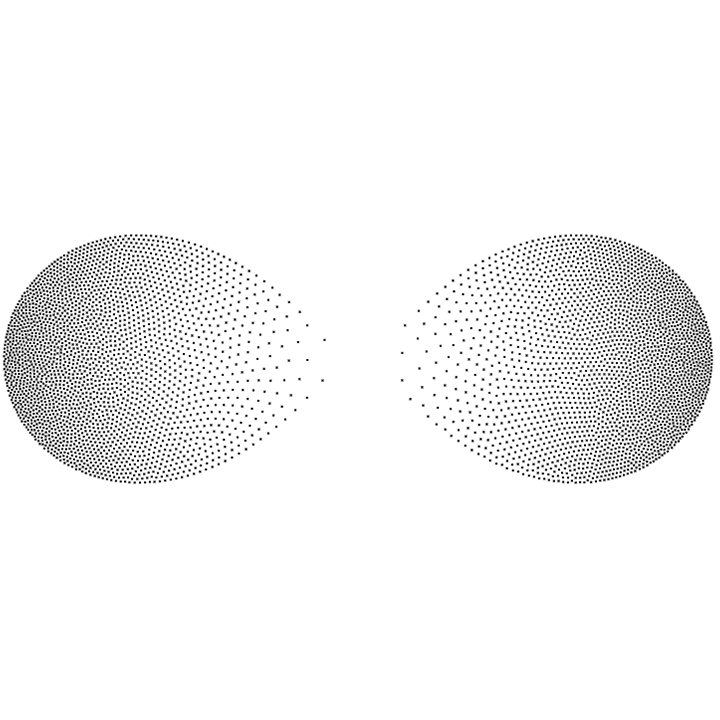}
		\end{center}
		\caption{$Q=|\zeta|^4-\frac 2 {\sqrt{2}}\re(\zeta^2)$}
	\end{subfigure}	
	\begin{subfigure}[h]{0.32\textwidth}
		\begin{center}
			\includegraphics[width=1.69in,height=1.69in]{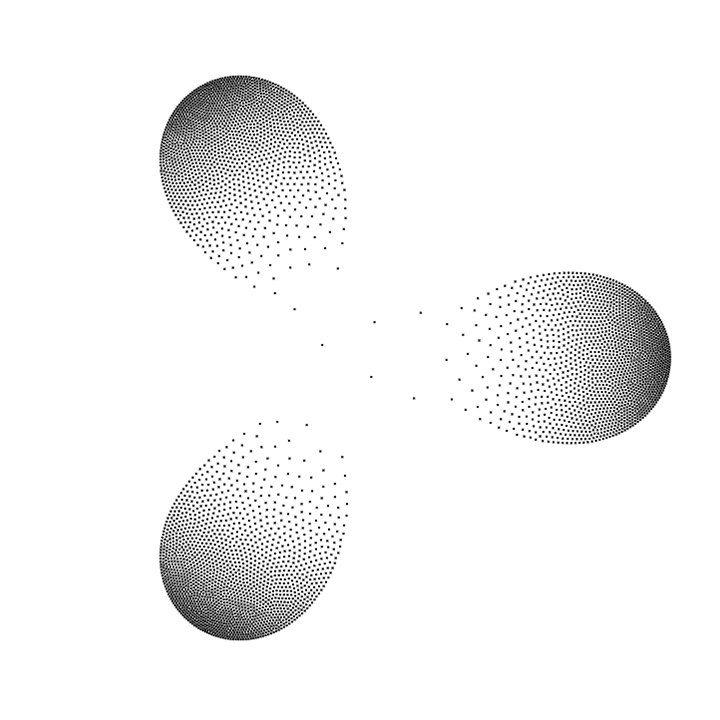}
		\end{center}
		\caption{$Q=|\zeta|^6-\frac 2 {\sqrt{5}}\re(\zeta^3)$}
	\end{subfigure}	
	\caption{Low-energy configurations with respect to various potentials $Q$.}
\label{fekete}
\end{figure}

Questions about freezing in Coulomb gas ensembles have been the subject of several investigations in the physics literature, cf.~e.g.~ the early works \cite{CLWH,DGIL} or the recent paper \cite{CSA} and the
extensive list of references there. Loosely speaking, one wants to understand as much as possible about the transition (as the inverse temperature $\beta\lr \infty$) between
an ``ordinary'' state of the Coulomb gas and a ``frozen'', presumably lattice-like state. As far as we are aware, the exact details of the transition remain largely unknown, and in particular
a basic question such as whether or not there exists a finite value $\beta_f<\infty$, such that the freezing takes place when $\beta$ increases beyond $\beta_f$, remains
an open question.

In \cite{CSA}, evidence is presented that a phase transition might occur at $\beta_f$ approximately equal to $70$. In this connection, we note that it is not
expected that ``perfect'' (or ``lattice-like'') freezing occurs at
this value $\beta_f$, but a rather different kind of phase transition, where the oscillations of the one-particle density near the boundary (the ``Hall effect'') start propagating inwards, from the boundary towards the bulk.
(We are grateful to Jean-Marie St\'{e}phan and to Paul Wiegmann for discussions concerning this point.)

The low temperature regime when $\beta_n$ increases at least logarithmically in $n$,
$\beta_n\ge c \log n$, was introduced in \cite{A2}. In this regime, we expect that a typical random configuration will look more and
more lattice-like as $c\lr\infty$, i.e., that we do have a perfect freezing in this transition.
 (Some examples of low-energy configurations, obtained numerically by an iterative method, are depicted in Figure \ref{fekete}.)

A glance at Figure \ref{fekete} gives the impression that different kinds of crystalline patterns seem to emerge. The most basic one is Abrikosov's triangular
lattice, which is believed to emerge close to points $p\in S$ at which the equilibrium density $\Delta Q(p)$ is \emph{strictly} positive. See Figure \ref{fig2}.
\begin{figure}[h]
\begin{center}
\includegraphics[width=.3\textwidth]{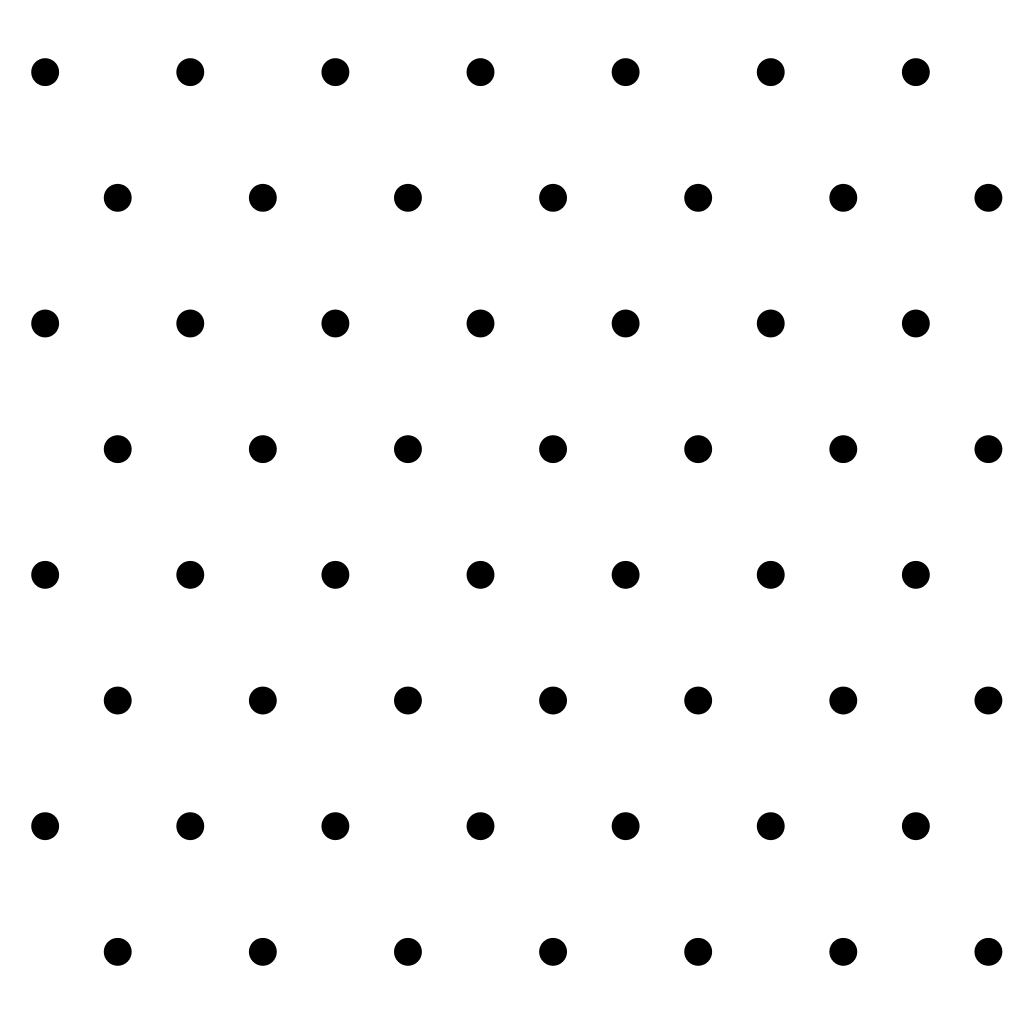}
\end{center}
\caption{Abrikosov's triangular lattice.}
\label{fig2}
\end{figure}

Likewise, other kinds of structures can be sensed from Figure \ref{fekete}, near \textit{singular} points $p\in S$ where the equilibrium density vanishes, i.e., $\Delta Q(p)=0$. Situation (B) depicts
a bulk singularity at $p=0$,
while (C),(D) have singularities on the boundary point $p=0$ (which in these cases are of ``lemniscate types'', see e.g.~\cite{BEG,GPSS} and references). In a rough sense (e.g.~\eqref{joh00})
the distribution is close to the equilibrium density also in the presence of singular points, but the exact details of the patterns which may emerge are not known to us.
However, for example the papers \cite{AKMW,AKS,AS,BEG,DS} deal with the corresponding $\beta=1$ ensembles.

As noted in \cite{A0}, the well-known ``Abrikosov conjecture'' as posed in \cite{AOC}, namely the problem of proving emergence of Abrikosov's lattice when rescaling Fekete configurations
about a ``regular'' point $p\in S$ where $\Delta Q(p)>0$, would follow if one could prove
a strong enough separation
of Fekete configurations $\family_n=\{\zeta_j\}_1^n$ as $n\lr\infty$.
(For example, in the Ginibre case $Q=|\zeta|^2$, proving $\liminf_{n\to\infty}\sep_n(\family_n)\ge 2^{1/2}3^{-1/4}$ would do.) It seems natural to add another layer to this problem and ask to what
extent Abrikosov's lattice emerges under the assumption $\beta_n\ge c\log n$, in the transition as $c\lr\infty$.

We finally offer a few brief remarks about some other works which are somewhat connected to the main theme in this note.

The counterpart to Theorem \ref{mth} (uniform separation) for Fekete configurations is well-known, and, apart from \cite{AOC}, is shown also in e.g.~\cite{LRY,RS} depending on an idea due to Lieb.

A somewhat weaker version of the equidistribution theorem (Theorem \ref{mth2}) for Fekete configurations  was shown in \cite{AOC,A0} using a variant of Landau's method which has been further
extended here. In particular those sources apply to all suitable families which obey certain sampling and interpolation conditions (a property that here is shown to hold almost surely for low temperature Coulomb ensembles).
The paper \cite{RS} suggests an utterly different approach, relying heavily on the minimum-energy property of Fekete configurations, and asserts that a discrepancy estimate
similar to \eqref{eq_d1} holds for bulk points with $\alpha=1$ in such a setting.

In the setting of $\beta$-ensembles, a recent result in \cite{ARSE} (part (2) of Theorem 1) provides
discrepancy estimates \eqref{eq_d1} with $\alpha$ close to 1. These are valid at any inverse temperature $\beta$, and provide failure probabilities for individual (deterministic) observation disks that are sufficiently away from the boundary of the droplet, and which may deteriorate as such limit is approached. With respect to separation, \cite[Theorem 1(4)]{ARSE} gives a local result in the bulk, which, when applied to the low temperature regime, asserts a similar order of separation as we obtain here. In contrast, our result applies globally to all points in the Coulomb gas, and without truncations that eliminate points close to the boundary. This is a nontrivial issue, since the
Hall effect postulates that the particle-distribution near the boundary is quite subtle when $\beta>1$. (In addition, \cite[Corollary 1.2]{ARSE} discusses
certain ``spatially averaged Coulomb-gases'' at low temperatures. As remarked in \cite[Paragraph below Corollary 1.2]{ARSE} these are different from the Coulomb gas ensembles, as considered here and
in Theorem 1 of \cite{ARSE}.)

The Coulomb gas on a sphere at a very low temperature ($\beta_n>n$) is studied in the paper \cite{BH}, where a certain Fekete-like behaviour is demonstrated. In this connection, it seems interesting to investigate the extent to which our present methods extend to Riemann-surfaces. We hope to return to this issue in a future work.

\end{document}